\documentclass[12pt,a4paper]{elsarticle}

\usepackage[english]{babel}
\usepackage{amsmath,amsthm,booktabs,hyperref}
\usepackage{amssymb}
\usepackage{graphicx}
\usepackage{xcolor}

\usepackage{tikz}
\usetikzlibrary{
	arrows,
	backgrounds,
	calc,
	decorations.pathmorphing,
	decorations.text,
	fixedpointarithmetic,
	patterns,
	patterns.meta,
	positioning,
	arrows.meta
}

\usepackage{subcaption}
\usepackage{refcount}
\usepackage{cleveref}
\usepackage{float}
\usepackage{placeins}

\usepackage[hmargin=2.5cm,vmargin=3cm]{geometry}

\DeclareCaptionLabelSeparator{mysep}{\hspace{0.5em}}
\DeclareCaptionLabelFormat{figsub}{Fig.~\arabic{figure}(#2)}

\crefname{subfigure}{Fig.}{Figs.}
\Crefname{subfigure}{Fig.}{Figs.}

\usepackage{letterswitharrows}

\newcommand{\myvec}[1]{\arrowoverset{#1}}
\newcommand{\mycev}[1]{\arrowoverset*{#1}}

\newtheorem{theorem}{Theorem}[section]
\newtheorem{proposition}[theorem]{Proposition}
\newtheorem{lemma}[theorem]{Lemma}
\newtheorem{observation}[theorem]{Observation}
\newtheorem{corollary}[theorem]{Corollary}

\theoremstyle{definition}

\newtheorem{definition}[theorem]{Definition}
\newtheorem{problem}[theorem]{Problem}

\newtheorem{remark}[theorem]{Remark}

\newcommand{\F}{\mathbb F}

\begin{document}

	\begin{frontmatter}
		\title{Orthogonal and unitary signings of cube-like graphs}
		
		\author[XUT]{Meirun Chen}
		\author[IRIF]{Reza Naserasr}
		
		\address[XUT]{School of Mathematics and Statistics, Xiamen University of Technology, Xiamen, Fujian 361024, China. mrchen@xmut.edu.cn.}
		\address[IRIF]{Université de Paris, CNRS, F-75013 Paris, France. reza@irif.fr}

\begin{abstract}
	A unitary signing of a $d$-regular graph is a Hermitian adjacency
	matrix $M$ whose nonzero entries lie in $\{\pm1,\pm i\}$ and satisfies
	$M^2=dI$. Motivated by the work of Alon and Zheng on orthogonal and
	unitary signings of cube-like graphs, we introduce the $\Theta$-property
	for a generating set $S\subseteq\mathbb Z_2^n$: whenever three pairwise
	disjoint subsets of $S$ have the same sum, at least two of them have
	even size.
	
	We prove that every zero-free generating set with the
	$\Theta$-property gives a cube-like graph $Q_S$ admitting a unitary
	signing, and we give an explicit local formula for such a signing. For
	Sidon sets, the $\Theta$-property is also necessary, yielding a
	characterization of the Sidon cube-like graphs that admit unitary
	signings. In this setting there are exactly $2^{|S|-n}$
	switching-equivalence classes of unitary signings, and we characterize
	when a unitary signing can be chosen to be orthogonal.
	
	The $\Theta$-property admits a linear-algebraic description in terms of
	the dependency space $\mathcal D$ of $S$:
	\[
	|D_1\cap D_2|
	\equiv
	|D_1||D_2|
	\pmod2
	\qquad
	(D_1,D_2\in\mathcal D).
	\]
	Equivalently, the map
	\[
	D\longmapsto \binom{|D|}{2}\pmod2
	\]
	is linear on $\mathcal D$. For the corresponding extremal set
	problem, in which zero is permitted, this formulation yields the sharp
	bound
	\(
	|S|\leq 2n+1
	\)
	for generating sets with the $\Theta$-property. We give elementary
	constructions attaining this bound for every $n\geq3$. Under the
	additional Sidon condition, the same extremal value is attained for
	every $n\geq10$ using binary self-dual codes. The exact Sidon maximum is
	also determined for $3\leq n\leq9$.
	
\end{abstract}

		
	\end{frontmatter}

	\section{Introduction}

	Several notions arising in signed graph theory, though seemingly independent, are deeply interconnected. In \cite{AACFMN13}, the authors studied graphs with $q(G)=2$, that is, graphs for which one can assign nonzero weights to the edges and choose values for the diagonal entries so that the resulting real symmetric matrix has only two distinct eigenvalues.
	Hypercubes are among their prominent examples, with weights $1$ or $-1$. 
	This signing subsequently appeared independently in the work of Huang \cite{H19} where he gave his now-famous proof of the sensitivity conjecture. The signed hypercubes in Huang's proof have also proved useful in other problems; see \cite{LNSW26} and the references therein.
	Another setting in which these ideas come together is the possible extension of the notion of strongly regular graphs to strongly regular signed graphs; see \cite{S19}, \cite{RB26} and the references therein.
	
	The signed-hypercube phenomenon naturally raises the question of how broadly it extends within the class of cube-like graphs. Alon and Zheng \cite{AZ20} considered such questions.
	In this work, we further investigate cube-like graphs that admit signings with only two distinct eigenvalues. We identify the $\Theta$-property as a general sufficient condition: whenever a generating set $S\subseteq\mathbb Z_2^n$ has the $\Theta$-property, the Cayley graph $Q_S=(\mathbb Z_2^n,S)$ admits a unitary signing. For Sidon sets this condition is also necessary, giving a complete characterization of the Sidon cube-like graphs that admit unitary signings; we also determine when the signing can be chosen to be orthogonal.
	
	The complex-signed graphs considered here may equivalently be viewed as gain graphs over the cyclic group of order four; the additive group $\mathbb Z_4=(\{0,1,2,3\},+)$ is isomorphic to the multiplicative group $\{1,i,-1,-i\}$ under the correspondence $k\mapsto i^k$. Cycle gains and their invariance under switching are standard in the theory of gain graphs; see, for example, \cite{Z89,CD22}. In the present work we use terminology modeled on signed graphs: cycles are classified as real or imaginary and, independently, as positive or negative.
	
	Before going further, we settle the notation and terminology in the next section. For some recent results on graphs with edge weightings yielding only two eigenvalues, we refer to \cite{AALMX26} and \cite{LMT26+}.

	\section{Signed and complex-signed graphs}
	
	A signed graph is a pair $(G, \sigma)$ consisting of a graph $G$ and an assignment $\sigma$ which assigns $+1$ or $-1$ to each edge of $G$. Switching at a vertex $v$ means multiplying the signs of all edges incident with $v$ by $-1$. For background on signed graphs in the context of this work, we refer to \cite{LNSW26} and references therein. 
	
	For the purposes of this paper, the following terminology is used. A \emph{complex-signed graph} is a pair $(G, \sigma)$ where $G$ is a mixed graph whose adjacencies may be edges (symmetric pairs) or arcs (non-symmetric pairs), and $\sigma$ assigns $+1,-1$ to the edges and $i,-i$ to the arcs with the convention that assigning $i$ to the arc $xy$ is equivalent to assigning $-i$ to $yx$. Thus one is free to flip the direction of an arc as long as the complex value assigned to it is conjugated. The key point here is that the adjacency matrix of $(G, \sigma)$ is then Hermitian, and thus has real eigenvalues. 
	
	We extend the notion of switching at $v$ to $(a)$-switching for each $a\in \{ +1, -1, i, -i\}$ as follows.
	
	\begin{itemize}
		\item $a=1$; edges and arcs incident with $v$ remain as they are and preserve their assignment.
		\item $a=-1$; edges and arcs incident with $v$ remain as they are, but the value assigned to them is negated (multiplied by $-1$).
		\item $a=i$; we first re-orient arcs incident with $v$ so that $v$ is the head of all the arcs at $v$, conjugating their assigned complex value if reoriented. Then we remove the orientations of all arcs with head $v$ and orient all $uv$ edges so that $v$ is their head. Then multiply each assigned value by $i$.
		\item $a=-i$; we apply both $(-1)$-switching and $(i)$-switching. 
	\end{itemize}

	In the context of signed graphs, a central concept is the sign of a cycle, which is invariant under switching. For complex-signed graphs, the cycle sign is defined analogously. Recall that a cycle is a sequence $v_1v_2\cdots v_k$, $k\geq3$, of distinct vertices such that $v_j$ is adjacent to $v_{j+1}$ for $1\leq j<k$, and $v_k$ is adjacent to $v_1$. The sequence $v_1v_kv_{k-1}\cdots v_2$ represents the same cycle, except that it is traversed in the opposite direction. The direction matters when defining the sign of a cycle $C$. Throughout this section, whenever we consider a cycle in a complex-signed graph, we regard it as equipped with one of its two possible directions of traversal. Thus, by a (directed) cycle we mean a cycle together with a chosen direction of traversal; we do not require the arcs of the mixed graph to be coherently oriented along the cycle. 
	
	\begin{definition}
		Given a cycle $C=v_1v_2\cdots v_k$ of a complex-signed graph $(G,\sigma)$, we define its sign as follows. First reverse every arc whose orientation disagrees with the direction of $C$, conjugating its assigned value. We then define $\sigma(C)=\prod_{j=1}^k \sigma(v_jv_{j+1})$ where $v_{k+1}=v_{1}$.  
	\end{definition}  
	
    The reversal of arcs in this definition is only a convenient way to write the formula; $\sigma(C)$ can equivalently be written as the product of the signs of the edges and arcs going along the cycle, conjugating the value when the direction of an arc does not match that of the cycle.
    The following is the key property of this definition.
	
	\begin{observation}
		Given a cycle $C$ of a complex-signed graph $(G,\sigma)$, an $a$-switching at a vertex, $a\in\{+1,-1,i,-i\}$, does not change the value of $\sigma(C)$.
	\end{observation}  
	
	Thus cycles in a complex-signed graph can be classified into four types as follows.
	
	A cycle is said to be an \emph{imaginary cycle} if its sign is either $i$ or $-i$. It is said to be \emph{real} if its sign is either $+1$ or $-1$. Similarly, a cycle is said to be \emph{positive} if its sign is one of $+1$ or $i$ and is said to be \emph{negative} if its sign is either $-1$ or $-i$. 
	
	Thus a cycle is imaginary if an odd number of its adjacencies are arcs. For imaginary cycles, the direction matters, meaning $\sigma(\myvec{C})=-\sigma(\mycev{C})$. In other words, an imaginary cycle is positive in one direction and negative in the other direction. A real cycle is either positive or negative regardless of the direction. 
	
	Two complex-signed graphs $(G, \sigma_1)$ and $(G,\sigma_2)$ are said to be \emph{equivalent} if one is obtained from the other by a series of switchings. 
	
	One of the basic facts about signed graphs is that the signs of cycles determine the equivalence class of signatures. The first case of this, namely, deciding if a signed graph is balanced, was proved by Harary \cite{H54}. The general case was proved by Zaslavsky in \cite{Z82}.
	
	\begin{theorem}
		Two signed graphs $(G,\sigma_1)$ and $(G, \sigma_2)$ are switching equivalent if and only if we have $\sigma_1(C)=\sigma_2(C)$ for every cycle $C$ of $G$. 
	\end{theorem}
	
	It follows from the proof of the theorem that in applying the theorem, we do not need to consider all cycles, but only a fundamental set of cycles. The corresponding statement for complex-signed graphs is a special case of the standard switching-equivalence criterion for gain graphs and is recorded here in the present terminology.
	
	\begin{theorem}\label{thm:SwitchingEquivalence}
		Two complex-signed graphs $(G,\sigma_1)$ and $(G, \sigma_2)$ are switching equivalent if and only if we have $\sigma_1(C)=\sigma_2(C)$ for every (directed) cycle $C$ of $G$. 
	\end{theorem}
	
	\begin{proof}
		The forward implication follows from the invariance of cycle signs under switching.
		
		For the converse, assume that $G$ is connected; otherwise, we may work on connected components separately. Hence we may consider a spanning tree $T$ of $G$. Rooting $T$ and applying switchings from the root downward, one can switch any signature on the tree so that all edges are assigned $+1$. Let $ \sigma'_1$ be the signature on $G$ obtained from $ \sigma_1$ by switching in such a way that $\sigma_1'(e)=+1$ for all $e\in E(T)$, and let $\sigma'_2$ be the signature obtained from $\sigma_2$ in the same way. 
		We now claim that $\sigma'_1=\sigma'_2$. Let $e\in E(G)$. If $e\in E(T)$, then $\sigma'_1(e)=\sigma'_2(e)=1$. If $e\not\in E(T)$, then $T+e$ contains a unique (directed) cycle $C_{e}$. Since all other edges of $C_e$ have sign $+1$ in both $(G, \sigma'_1)$ and $(G, \sigma'_2)$ we have $\sigma'_1(C_e)=\sigma'_1(e)$ and $\sigma'_2(C_e)=\sigma'_2(e)$. However, we have $\sigma'_1(C_e)=\sigma_1(C_e)$ and $\sigma'_2(C_e)=\sigma_2(C_e)$ because switching preserves cycle signs, and we have $\sigma_1(C_e)=\sigma_2(C_e)$ by the assumption. Therefore, $\sigma'_1(e)=\sigma'_2(e)$.
	\end{proof}

		Observe that the proof uses only a fundamental set of cycles, namely the cycles $C_e$ for $e\in E(G)\setminus E(T)$. Suppose that a cycle $C$ is the symmetric difference of two cycles $C_1$ and $C_2$ from this fundamental set. Choose directions on $C_1$ and $C_2$ so that every edge common to them is traversed in opposite directions. The remaining edges then induce a direction on $C$, and for the corresponding (directed) cycles we have
$\sigma(\myvec{C})=\sigma(\myvec{C_1})\sigma(\myvec{C_2}).$
Indeed, on each common edge the two contributions are conjugates of one another and hence cancel in the product.

	\section{Cube-like graphs}
	
	A binary Cayley graph is a graph with vertex set $\mathbb{Z}_2^n$ and generating set $S\subseteq\mathbb{Z}_2^n$, where $xy$ is an edge if and only if $x+y\in S$. In the graph-theoretic parts of the paper we assume $0\notin S$, so that $Q_S$ is loopless; in the extremal set problem considered later, $0$ is allowed unless stated otherwise. When $n$ is clear from the context, the cube-like graph corresponding to $S\subset \mathbb{Z}_2^n$ is denoted by $Q_S$. The graphs are named cube-like by Lov\'asz \cite{L71} because when $S$ is a (standard) basis of  $\mathbb{Z}_2^n$, then $Q_S$ is the hypercube of dimension $n$ (or isomorphic to it). The $n$-dimensional cube is denoted by $H_n$.
	
	Following Huang's proof of the sensitivity conjecture, a question of interest is which cube-like graphs admit an edge-weighting such that the resulting adjacency matrix has only two eigenvalues. This has been studied, among others, in \cite{AZ20} where complex numbers are also considered as weights but under the condition that the adjacency matrix is Hermitian, so that the eigenvalues are real.
	
	Observe that in a cube-like graph, if $x$ and $y$ are connected by a 2-path, whose edges are labeled $s_1$ (incident to $x$) and $s_2$ (incident to $y$), then there is another path connecting them whose edges are labeled  $s_2$ (incident to $x$) and $s_1$ (incident to $y$). The middle vertex of the first is $x+s_1=y+s_2$ and that of the second is $x+s_2=y+s_1$. We give special attention to cube-like graphs in which there are at most two 2-paths between any two vertices. A subset $S\subseteq\mathbb{Z}_2^n$ is said to be a Sidon set if the sum of any two distinct elements of $S$ is distinct from the sum of any other two elements of $S$. Formally, 
	$a+b=c+d,$ implies $\{a,b\}=\{c,d\}$, where $a,b,c,d\in S$ and $a\neq b, c\neq d$.

	 A Sidon cube-like graph is a cube-like graph whose generating set is a Sidon set. The following proposition explains why these graphs are of particular interest.
	
	\begin{proposition}\label{prop:A2-4Cycles}
		A Sidon cube-like graph $G=Q_S$ with a signing $\sigma$ satisfies $A_{(G,\sigma)}^2=|S|I$ if and only if $\sigma(C)=-1$ for every $4$-cycle $C$ of $G$. 
	\end{proposition}
	
	\begin{proof}
	Each $x-y-x$ walk of length $2$ contributes $1$ to the $(x,x)$-entry of $A_{(G,\sigma)}^2$, so the diagonal entries of $A_{(G,\sigma)}^2$ are all $|S|$. The only other walks of length 2 are between distinct vertices with a common neighbor. Since $S$ is a Sidon set, for any pair $x,y$ of distinct vertices with a common neighbor, there are exactly two $2$-paths $x-z-y$ and $x-t-y$. The $(x,y)$-entry of $A_{(G,\sigma)}^2$ then is $\sigma(xz)\sigma(zy)+\sigma(xt)\sigma(ty)$, which must be 0. By switchings at $x$ and $y$, if necessary, we may assume  $\sigma(xz)=\sigma(zy)=1$. Then for the entry to be $0$, we must have $\sigma(xt)\sigma(ty)=-1$. Hence the 4-cycle $xtyz$, in this cyclic direction, has sign $-1$. Finally, recall that switching does not change signs of cycles.
	\end{proof}

	Consequently, if $\sigma$ is a unitary signing of a Sidon cube-like graph $Q_S$, then its restriction to any subgraph $H'$ of $Q_S$ isomorphic to $H_j$, for some $j\leq n$, is a unitary signing of $H'$.

	\section{Orthogonal and unitary cube-like graphs}
	
	    Following the terminology of Alon and Zheng \cite{AZ20}, a complex-signed adjacency matrix $M$ of a $d$-regular graph $G$ is called a \emph{unitary signing} when any two distinct rows are orthogonal with respect to the usual Hermitian inner product. Equivalently,
	\[
	M\overline{M}^{\,T}=dI.
	\]
	Since $M$ is Hermitian, this is equivalent to $M^2=dI$; hence $M/\sqrt d$ is a unitary matrix. If all edge values are real, the signing is called \emph{orthogonal}, and $M/\sqrt d$ is then an orthogonal matrix. The underlying matrix condition is classical: real weighing matrices and their complex analogues were studied, for example, by Seberry Wallis \cite{SW72} and by Seberry and Whiteman \cite{SW80}, respectively. In the signed-graph setting, Ramezani studied signed adjacency matrices with two distinct eigenvalues \cite{R20}.

	Alon and Zheng \cite{AZ20} provided families of orthogonal and unitary cube-like graphs and proved the nonexistence of such signatures on some other cube-like graphs. The full classification of which cube-like graphs admit such a signature remains open. Here we first give a general sufficient condition: every generating set $S$ with the $\Theta$-property gives a cube-like graph $Q_S$ admitting a unitary signing. For Sidon sets we show that the condition is also necessary, and hence gives a full characterization. Besides this characterization, our approach has several additional advantages. For Sidon sets we characterize all such signings up to switching; in every affirmative case, we give a local formula for the signing, so that $\sigma(xy)$ can be computed 	directly from the edge $xy$; when a Sidon set $S$ gives a cube-like graph $Q_S$ with no unitary signing, we exhibit a subgraph on $O(n^2)$ vertices on which it is already impossible to assign signs so that every two-label $4$-cycle has sign $-1$; and, for Sidon sets, we determine precisely when a unitary signing can be chosen to be orthogonal.

	Orthogonal signings of the hypercube are often described globally (noting that they are all switching equivalent). A particular signature that is commonly considered is the one based on the inductive construction of $H_n$ as the Cartesian product $H_{n-1} \square K_2$. In this definition the edges corresponding to $K_2$ induce a matching, thus an acyclic subgraph. After a switching so that the sign of each of these edges is $+1$, one concludes that the signature on one copy of $H_{n-1}$ is the negation of the signature on the other copy. This gives an inductive definition of such a signature. A local description of this particular signature is as follows.
	
	Given a binary vector $x$, let $b_{j}(x)$ be the number of coordinates preceding the $j$th coordinate whose values are $1$. Let $xy$ be an edge of the hypercube $H_n$ and assume they differ in coordinate $j$. We define \[\sigma_n(xy)=(-1)^{b_j(x)}.\] Since $x$ and $y$ differ only in coordinate $j$, we have $b_j(x)=b_j(y)$ and hence $\sigma_n(xy)$ is defined independently of the order of $x$ and $y$. This signing is the same as the inductive signing described above. However, we directly prove the key property of this signature in the following proposition.
	
	\begin{proposition}\label{prop:4Cycles-1}
		For any 4-cycle $C$ of $(H_n, \sigma_n)$ we have $\sigma_n(C)=-1$.
	\end{proposition}  
	
	\begin{proof}
		Let $x_1x_2x_3x_4$ be a labeling of its vertices. Then $x_1$ and $x_2$ differ in the same coordinate in which $x_3$ and $x_4$ differ. Let $k$ be this coordinate. Similarly, the difference between $x_1$ and $x_4$ is in the same coordinate as the difference of $x_2$ and $x_3$, say coordinate $l$. Without loss of generality, assume $k<l$. Then it follows from the formula for $\sigma_n$ that $\sigma_n(x_1x_2)=\sigma_n(x_3x_4)$ and that $\sigma_n(x_1x_4)=-\sigma_n(x_2x_3)$. Thus the product of the four values is $-1$. 
	\end{proof}

	For most of the paper, we consider connected cube-like graphs. Equivalently, the generating set $S$ generates the group $\mathbb{Z}_2^n$ on which $G$ is built. Choosing a basis contained in $S$, we see that $Q_S$ contains an isomorphic copy of $H_n$, the hypercube of dimension $n$. After applying an automorphism of $\mathbb{Z}_2^n$, we may assume that $Q_S$ contains a standard copy of $H_n$ as a subgraph, equivalently we assume that $S$ contains the standard basis: $\{e_1,e_2, \dots, e_n\}\subseteq S$. 
	
	For the rest of this section, unless stated otherwise, we assume that $Q_S$ is a cube-like graph with a unitary signing $\sigma$ whose restriction to a spanning copy of $H_n$ is also unitary. We record three points. First, if $S$ is Sidon, every spanning copy of $H_n$ obtained from a basis contained in $S$ has this property. Second, a unitary signing of a general cube-like graph need not restrict to a unitary signing on any spanning copy of $H_n$; see \Cref{fig:PC(3)}. Third, after applying suitable switchings, we may assume that the restriction of $\sigma$ to $E(H_n)$ is the signing $\sigma_n$ defined by $\sigma_n(xy)=(-1)^{b_j(x)}$. Indeed, both signings assign $-1$ to every $4$-cycle of $H_n$. Since the sign of any (directed) cycle of $H_n$ is the product of the signs of the (directed) $4$-cycles that generate it, we have $\sigma(C)=\sigma_n(C)$ for every cycle $C$ of $H_n$. Hence, by \Cref{thm:SwitchingEquivalence}, $\sigma_n$ is switching equivalent to $\sigma|_{_{E(H_n)}}$.

	\begin{lemma}\label{lem:edge-type}
		Assume that $S$ is Sidon, that $\sigma$ is a unitary signing of the binary Cayley graph $Q_S$, that $S$ contains the standard basis, and that the restriction of $\sigma$ to the subgraph $H_n$ is $\sigma_n$. Then, for each $s\in S$, either all edges corresponding to $s$ are assigned real values ($1$ or $-1$), or they are all assigned imaginary values ($i$ or $-i$).
	\end{lemma} 
	
	\begin{proof}
		Consider a 4-cycle $x_1x_2x_3x_4$. 
		If \(x_1x_2\) and \(x_3x_4\) are both real, then either \(x_1x_4\) and \(x_2x_3\) are both real, or they are both imaginary.
		To complete the proof, note that the assumption $\left.\sigma\right|_{_{E(H_n)}}=\sigma_n$ implies that on the edges of $H_n$ there are only real values. For $s\in S\setminus\{e_1,e_2,\dots,e_n\}$, suppose that $xy$ and $zt$ are two edges corresponding to $s$. Consider an $(x-z)$-path $x = x_1 - x_2 - \cdots - x_k = z$ in $H_n$ and the parallel $(y-t)$-path $y = y_1 - y_2 - \cdots - y_k = t$. Here parallel means $x_j+x_{j+1}=y_j+y_{j+1}$. Thus all edges $x_jy_j$ correspond to $s$. The claim follows by applying the preceding observation successively to these $4$-cycles.	   
	\end{proof}
	
	\begin{definition}
		An \emph{antipodal $2n$-cycle} in $H_n$ is a cycle whose edges, in cyclic order, are labeled
		\[
		e_{\pi(1)},e_{\pi(2)},\ldots,e_{\pi(n)},e_{\pi(1)},e_{\pi(2)},\ldots,e_{\pi(n)}
		\]
		for some permutation $\pi$. Equivalently, each pair of antipodal vertices of the cycle is also an antipodal pair in $H_n$. (See \Cref{fig:Antipodal2nCycle} for a depiction.)
	\end{definition} 
	
	\begin{figure}[htbp]
		\centering
		\begin{tikzpicture}[
			scale=0.9, 
			every node/.style={font=\fontfamily{cmr}\selectfont\footnotesize}, 
			vertex/.style={circle, draw, fill=white, inner sep=1.8pt, minimum size=4.5pt},
			edge label/.style={color=purple, font=\fontfamily{cmr}\selectfont\tiny},
			vlabel/.style={color=black, font=\fontfamily{cmr}\selectfont\tiny, inner sep=5pt}
			]
			
			\def\a{2.5}
			\def\b{2.8}
			
			\node[vertex] (v0) at ({ \a*cos(90) }, { \b*sin(90) }) {};
			\node[vertex] (v1) at ({ \a*cos(54) }, { \b*sin(54) }) {};
			\node[vertex] (v2) at ({ \a*cos(18) }, { \b*sin(18) }) {};
			\node[vertex] (v3) at ({ \a*cos(-18) }, { \b*sin(-18) }) {};
			\node[vertex] (v4) at ({ \a*cos(-54) }, { \b*sin(-54) }) {};
			\node[vertex] (v5) at ({ \a*cos(-90) }, { \b*sin(-90) }) {};
			\node[vertex] (v6) at ({ \a*cos(-126) }, { \b*sin(-126) }) {};
			\node[vertex] (v7) at ({ \a*cos(-162) }, { \b*sin(-162) }) {};
			\node[vertex] (v8) at ({ \a*cos(162) }, { \b*sin(162) }) {};
			\node[vertex] (v9) at ({ \a*cos(126) }, { \b*sin(126) }) {};
			
			\node[vlabel, rotate=90, anchor=west] at (v0) {$00\cdots0$};
			\node[vlabel, rotate=54, anchor=west] at (v1) {$10\cdots0$};
			\node[vlabel, rotate=18, anchor=west] at (v2) {$110\cdots0$};
			\node[vlabel, rotate=-18, anchor=west] at (v3) {$11\cdots100$};
			\node[vlabel, rotate=-54, anchor=west] at (v4) {$11\cdots10$};
			\node[vlabel, rotate=90, anchor=east] at (v5) {$11\cdots1$};
			\node[vlabel, rotate=54, anchor=east] at (v6) {$01\cdots1$};
			\node[vlabel, rotate=18, anchor=east] at (v7) {$001\cdots1$};
			\node[vlabel, rotate=-18, anchor=east] at (v8) {$00\cdots011$};
			\node[vlabel, rotate=-54, anchor=east] at (v9) {$00\cdots01$};
			
			\draw (v0) to [bend left=16] node[midway, sloped, above, edge label] {$e_1$} (v1);
			\draw (v1) to [bend left=16] node[midway, sloped, above, edge label] {$e_2$} (v2);
			
			\draw (v3) to [bend left=16] node[midway, sloped, below, edge label] {$e_{n-1}$} (v4);
			
			\draw (v4) to [bend left=16] node[midway, sloped, below, edge label] {$e_n$} (v5);
			
			\draw (v5) to [bend left=16] node[midway, sloped, below, edge label] {$e_1$} (v6);
			\draw (v6) to [bend left=16] node[midway, sloped, below, edge label] {$e_2$} (v7);
			\draw (v8) to [bend left=16] node[midway, sloped, above, edge label] {$e_{n-1}$} (v9);
			\draw (v9) to [bend left=16] node[midway, sloped, above, edge label] {$e_n$} (v0);
			
			\draw[black, loosely dotted, line width=1.2pt] (v2) to [bend left=16] (v3);
			\draw[black, loosely dotted, line width=1.2pt] (v7) to [bend left=16] (v8);
			
		\end{tikzpicture}
		\caption{An antipodal $2n$-cycle in $H_n$.}
		\label{fig:Antipodal2nCycle}
	\end{figure}
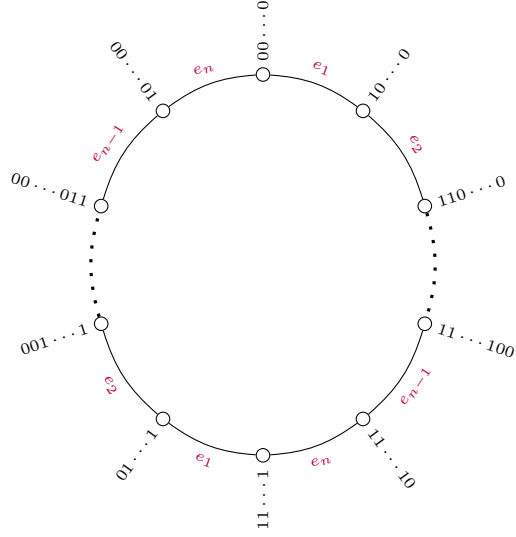

		A key point in our proofs is the sign of an antipodal $2n$-cycle, which is determined by the following lemma.
	
	\begin{lemma}\label{lem:AntipodalCycle}
		Any antipodal $2n$-cycle $C$ of $H_n$ is the binary sum of $\tbinom{n}{2}$ $4$-cycles. 
	\end{lemma} 
	
	\begin{proof}
		We apply induction on $n$. The statement is clear for $H_2$ since the graph itself is both a 4-cycle and an antipodal $4$-cycle. Assuming the statement holds for $H_n$, in the construction of $H_{n+1}$ as the Cartesian product $H_n \square K_2$ an antipodal $(2n+2)$-cycle $C$ lies inside a $C'\square K_2$ where all the edges corresponding to $K_2$ are labeled $e_{n+1}$ and $C'$ is an antipodal $2n$-cycle of $H_n$ (see \Cref{subfig:UnfilledLabels}). Thus $C$ can be written as the binary sum of $C'$ with $n$ 4-cycles as shown in \Cref{subfig:PatternedFills}. 
	\end{proof}
	
	\begin{remark} To present $C$ together with all the $4$-cycles that generate it (those used in the proof above) we need $\binom{n+1}{2}+1$ vertices of $H_n$. 
	\end{remark}
	
	\begin{corollary}\label{coro:SignOfAntipodalCycle}
		Let $C$ be an antipodal $2n$-cycle of $H_n$. Then,
		\begin{itemize}
			\item $\sigma_n(C) = -1$ if $n \equiv 2$ or $3 \pmod 4$,
			\item $\sigma_n(C) = +1$ if $n \equiv 0$ or $1 \pmod 4$.
		\end{itemize}
	\end{corollary}

	\begin{figure}[H]
		\centering
		
		\begin{subfigure}[b]{0.49\textwidth}
			\centering
			\begin{tikzpicture}[
				scale=0.9, 
				every node/.style={font=\fontfamily{cmr}\selectfont\footnotesize}, 
				vertex/.style={circle, draw, fill=white, inner sep=1.8pt, minimum size=4.5pt},
				edge label/.style={color=purple, inner sep=1.5pt, font=\fontfamily{cmr}\selectfont\tiny}, 
				bg edge/.style={draw=black!40, thin},
				bg dotted/.style={draw=black!40, loosely dotted, line width=1pt},
				bg product/.style={draw=black!50, dashed, thick}, 
				highlight/.style={draw=black, line width=1.6pt},              
				highlight dotted/.style={draw=black, loosely dotted, line width=1.6pt},
				highlight product/.style={draw=black, dashed, line width=1.6pt}
				]
				
				\def\a{2.5}
				\def\b{2.8}
				\def\dx{1.5}
				\def\dy{1.2}
				
				\node[vertex] (v0b) at ({ \a*cos(90) + \dx }, { \b*sin(90) + \dy }) {};
				\node[vertex] (v1b) at ({ \a*cos(54) + \dx }, { \b*sin(54) + \dy }) {};
				\node[vertex] (v2b) at ({ \a*cos(18) + \dx }, { \b*sin(18) + \dy }) {};
				\node[vertex] (v3b) at ({ \a*cos(-18) + \dx }, { \b*sin(-18) + \dy }) {};
				\node[vertex] (v4b) at ({ \a*cos(-54) + \dx }, { \b*sin(-54) + \dy }) {};
				\node[vertex] (v5b) at ({ \a*cos(-90) + \dx }, { \b*sin(-90) + \dy }) {};
				\node[vertex] (v6b) at ({ \a*cos(-126) + \dx }, { \b*sin(-126) + \dy }) {};
				\node[vertex] (v7b) at ({ \a*cos(-162) + \dx }, { \b*sin(-162) + \dy }) {};
				\node[vertex] (v8b) at ({ \a*cos(162) + \dx }, { \b*sin(162) + \dy }) {};
				\node[vertex] (v9b) at ({ \a*cos(126) + \dx }, { \b*sin(126) + \dy }) {};
				
				\node[vertex] (v0f) at ({ \a*cos(90) }, { \b*sin(90) }) {};
				\node[vertex] (v1f) at ({ \a*cos(54) }, { \b*sin(54) }) {};
				\node[vertex] (v2f) at ({ \a*cos(18) }, { \b*sin(18) }) {};
				\node[vertex] (v3f) at ({ \a*cos(-18) }, { \b*sin(-18) }) {};
				\node[vertex] (v4f) at ({ \a*cos(-54) }, { \b*sin(-54) }) {};
				\node[vertex] (v5f) at ({ \a*cos(-90) }, { \b*sin(-90) }) {};
				\node[vertex] (v6f) at ({ \a*cos(-126) }, { \b*sin(-126) }) {};
				\node[vertex] (v7f) at ({ \a*cos(-162) }, { \b*sin(-162) }) {};
				\node[vertex] (v8f) at ({ \a*cos(162) }, { \b*sin(162) }) {};
				\node[vertex] (v9f) at ({ \a*cos(126) }, { \b*sin(126) }) {};
				
				\draw[bg edge] (v0f) to [bend left=16] node[midway, sloped, above, edge label] {$e_1$} (v1f);
				\draw[bg edge] (v1f) to [bend left=16] node[midway, sloped, above, edge label] {$e_2$} (v2f);
				\draw[bg dotted] (v2f) to [bend left=16] (v3f);
				\draw[bg edge] (v3f) to [bend left=16] (v4f);
				\draw[bg edge] (v4f) to [bend left=16] node[midway, sloped, below, edge label] {$e_n$} (v5f);
				
				\draw[highlight] (v5f) to [bend left=16] node[midway, sloped, below, edge label] {$e_1$} (v6f);
				\draw[highlight] (v6f) to [bend left=16] node[midway, sloped, below, edge label] {$e_2$} (v7f);
				\draw[highlight dotted] (v7f) to [bend left=16] (v8f);
				\draw[highlight] (v8f) to [bend left=16] node[midway, sloped, above, edge label] {$e_{n-1}$} (v9f);
				\draw[highlight] (v9f) to [bend left=16] (v0f);
				
				\draw[highlight] (v0b) to [bend left=16] (v1b);
				\draw[highlight] (v1b) to [bend left=16] (v2b);
				\draw[highlight dotted] (v2b) to [bend left=16] (v3b);
				\draw[highlight] (v3b) to [bend left=16] (v4b);
				\draw[highlight] (v4b) to [bend left=16] (v5b);
				
				\draw[bg edge] (v5b) to [bend left=16] (v6b);
				\draw[bg edge] (v6b) to [bend left=16] (v7b);
				\draw[bg dotted] (v7b) to [bend left=16] (v8b);
				\draw[bg edge] (v8b) to [bend left=16] (v9b);
				\draw[bg edge] (v9b) to [bend left=16] (v0b);
				
				\draw[highlight product] (v0f) -- (v0b); 
				
				\draw[highlight product] (v5f) -- (v5b) node[pos=0.35, sloped, above, edge label] {$e_{n+1}$};
				
				\draw[bg product] (v1f) -- (v1b);
				\draw[bg product] (v2f) -- (v2b);
				\draw[bg product] (v3f) -- (v3b);
				\draw[bg product] (v4f) -- (v4b);
				
				\draw[bg product] (v6f) -- (v6b) node[midway, sloped, above, edge label] {$e_{n+1}$};
				\draw[bg product] (v7f) -- (v7b) node[midway, sloped, above, edge label] {$e_{n+1}$};
				
				\draw[bg product] (v8f) -- (v8b);
				\draw[bg product] (v9f) -- (v9b);
			\end{tikzpicture}
			\caption{$C_{2n}\square K_2$ inside $H_{n+1}$;}
			\label{subfig:UnfilledLabels}
		\end{subfigure}
		\begin{subfigure}[b]{0.49\textwidth}
			\centering
			\begin{tikzpicture}[
				scale=0.9, 
				every node/.style={font=\fontfamily{cmr}\selectfont\footnotesize}, 
				vertex/.style={circle, draw, fill=white, inner sep=1.8pt, minimum size=4.5pt},
				edge label/.style={color=blue, inner sep=2pt},
				bg edge/.style={draw=gray!50, thin},
				bg dotted/.style={draw=gray!50, loosely dotted, line width=1pt},
				bg product/.style={draw=gray!70, dashed, thick}, 
				highlight/.style={draw=red, line width=1.6pt},
				highlight dotted/.style={draw=red, loosely dotted, line width=1.6pt},
				highlight product/.style={draw=red, dashed, line width=1.6pt}
				]
				
				\def\a{2.5}
				\def\b{2.8}
				\def\dx{1.5}
				\def\dy{1.2}
				
				\node[vertex] (v0b) at ({ \a*cos(90) + \dx }, { \b*sin(90) + \dy }) {};
				\node[vertex] (v1b) at ({ \a*cos(54) + \dx }, { \b*sin(54) + \dy }) {};
				\node[vertex] (v2b) at ({ \a*cos(18) + \dx }, { \b*sin(18) + \dy }) {};
				\node[vertex] (v3b) at ({ \a*cos(-18) + \dx }, { \b*sin(-18) + \dy }) {};
				\node[vertex] (v4b) at ({ \a*cos(-54) + \dx }, { \b*sin(-54) + \dy }) {};
				\node[vertex] (v5b) at ({ \a*cos(-90) + \dx }, { \b*sin(-90) + \dy }) {};
				\node[vertex] (v6b) at ({ \a*cos(-126) + \dx }, { \b*sin(-126) + \dy }) {};
				\node[vertex] (v7b) at ({ \a*cos(-162) + \dx }, { \b*sin(-162) + \dy }) {};
				\node[vertex] (v8b) at ({ \a*cos(162) + \dx }, { \b*sin(162) + \dy }) {};
				\node[vertex] (v9b) at ({ \a*cos(126) + \dx }, { \b*sin(126) + \dy }) {};
				
				\node[vertex] (v0f) at ({ \a*cos(90) }, { \b*sin(90) }) {};
				\node[vertex] (v1f) at ({ \a*cos(54) }, { \b*sin(54) }) {};
				\node[vertex] (v2f) at ({ \a*cos(18) }, { \b*sin(18) }) {};
				\node[vertex] (v3f) at ({ \a*cos(-18) }, { \b*sin(-18) }) {};
				\node[vertex] (v4f) at ({ \a*cos(-54) }, { \b*sin(-54) }) {};
				\node[vertex] (v5f) at ({ \a*cos(-90) }, { \b*sin(-90) }) {};
				\node[vertex] (v6f) at ({ \a*cos(-126) }, { \b*sin(-126) }) {};
				\node[vertex] (v7f) at ({ \a*cos(-162) }, { \b*sin(-162) }) {};
				\node[vertex] (v8f) at ({ \a*cos(162) }, { \b*sin(162) }) {};
				\node[vertex] (v9f) at ({ \a*cos(126) }, { \b*sin(126) }) {};
				
				\begin{pgfonlayer}{background}
					\fill[pattern={Lines[angle=45, distance=1.6pt, line width=0.35pt]}, pattern color=red!25] 
					({ \a*cos(90) }, { \b*sin(90) }) to [bend left=16] ({ \a*cos(54) }, { \b*sin(54) }) -- 
					({ \a*cos(54) + \dx }, { \b*sin(54) + \dy }) to [bend right=16] ({ \a*cos(90) + \dx }, { \b*sin(90) + \dy }) -- cycle;
					
					\fill[pattern={Lines[angle=-45, distance=1.6pt, line width=0.35pt]}, pattern color=blue!25] 
					({ \a*cos(54) }, { \b*sin(54) }) to [bend left=16] ({ \a*cos(18) }, { \b*sin(18) }) -- 
					({ \a*cos(18) + \dx }, { \b*sin(18) + \dy }) to [bend right=16] ({ \a*cos(54) + \dx }, { \b*sin(54) + \dy }) -- cycle;
					
					\fill[pattern={Lines[angle=90, distance=1.6pt, line width=0.35pt]}, pattern color=orange!45] 
					({ \a*cos(-18) }, { \b*sin(-18) }) to [bend left=16] ({ \a*cos(-54) }, { \b*sin(-54) }) -- 
					({ \a*cos(-54) + \dx }, { \b*sin(-54) + \dy }) to [bend right=16] ({ \a*cos(-18) + \dx }, { \b*sin(-18) + \dy }) -- cycle;
					
					\fill[pattern={Lines[angle=0, distance=1.6pt, line width=0.35pt]}, pattern color=green!35] 
					({ \a*cos(-54) }, { \b*sin(-54) }) to [bend left=16] ({ \a*cos(-90) }, { \b*sin(-90) }) -- 
					({ \a*cos(-90) + \dx }, { \b*sin(-90) + \dy }) to [bend right=16] ({ \a*cos(-54) + \dx }, { \b*sin(-54) + \dy }) -- cycle;
					
					\draw[teal, thin, opacity=0.4] (0,0) ellipse (2.35cm and 2.64cm);
				\end{pgfonlayer}
				
				\draw[bg edge] (v0f) to [bend left=16] (v1f);
				\draw[bg edge] (v1f) to [bend left=16] (v2f);
				\draw[bg dotted] (v2f) to [bend left=16] (v3f);
				\draw[bg edge] (v3f) to [bend left=16] (v4f);
				\draw[bg edge] (v4f) to [bend left=16] (v5f);
				
				\draw[highlight] (v5f) to [bend left=16] (v6f);
				\draw[highlight] (v6f) to [bend left=16] (v7f);
				\draw[highlight dotted] (v7f) to [bend left=16] (v8f);
				\draw[highlight] (v8f) to [bend left=16] (v9f);
				\draw[highlight] (v9f) to [bend left=16] (v0f);
				
				\draw[highlight] (v0b) to [bend left=16] (v1b);
				\draw[highlight] (v1b) to [bend left=16] (v2b);
				\draw[highlight dotted] (v2b) to [bend left=16] (v3b);
				\draw[highlight] (v3b) to [bend left=16] (v4b);
				\draw[highlight] (v4b) to [bend left=16] (v5b);
				
				\draw[bg edge] (v5b) to [bend left=16] (v6b);
				\draw[bg edge] (v6b) to [bend left=16] (v7b);
				\draw[bg dotted] (v7b) to [bend left=16] (v8b);
				\draw[bg edge] (v8b) to [bend left=16] (v9b);
				\draw[bg edge] (v9b) to [bend left=16] (v0b);
				
				\draw[highlight product] (v0f) -- (v0b);
				\draw[highlight product] (v5f) -- (v5b);
				
				\draw[bg product] (v1f) -- (v1b);
				\draw[bg product] (v2f) -- (v2b);
				\draw[bg product] (v3f) -- (v3b);
				\draw[bg product] (v4f) -- (v4b);
				\draw[bg product] (v6f) -- (v6b);
				\draw[bg product] (v7f) -- (v7b);
				\draw[bg product] (v8f) -- (v8b);
				\draw[bg product] (v9f) -- (v9b);
			\end{tikzpicture}
			\caption{Forming {\color{red} $(2n+2)$-cycle} from {\color{teal} $2n$-cycle}.}
			\label{subfig:PatternedFills}
		\end{subfigure}
	\end{figure}

	Next we consider a specific cube-like graph, namely the projective cube of dimension $n$, denoted $PC(n)$. This is the cube-like graph with $S=\{e_1,e_2,\dots,e_n,J\}$, where $J=11\cdots1$. Equivalently, $PC(n)$ is the antipodal projection of $H_{n+1}$; it is also known as a folded cube. For more on this family, we refer to \cite{CNS26} and references therein. We assume $\varsigma_n$ is a complex signing of $PC(n)$ whose restriction to $H_n$ is $\sigma_n$. Under this assumption, we show that the values of $\varsigma_n$ on the edges corresponding to $J$ are almost uniquely determined.

	\begin{theorem}\label{thm:Unitary-PC(k)}
		For $n\geq4$, the projective cube $PC(n)$ has exactly two switching-equivalence classes of unitary signings, represented by $\varsigma_n$ and $\varsigma_n^J$, whose restrictions to $H_n$ are both $\sigma_n$. Moreover, on the edges corresponding to $J$, the signing $\varsigma_n^J$ is the negation of $\varsigma_n$; that is,
		\(
		\varsigma_n(x\bar{x})=-\varsigma_n^J(x\bar{x}).
		\)
		Furthermore,
		\begin{itemize}
			\item $\varsigma_n(x\bar{x}) \in \{+1,-1 \}$ if $n \equiv 0$ or $3 \pmod 4$;
			\item $\varsigma_n(x\bar{x}) \in \{+i,-i\}$ if $n \equiv 1$ or $2 \pmod 4$.
		\end{itemize}
	\end{theorem}

	\begin{proof}
		For each edge $xy$ with $x+y=e_j$ we take $\varsigma_n(xy)=\sigma_n(xy)=(-1)^{b_j(x)}$. To define $\varsigma_n(x\bar{x})$, we take $O(x)$ to be the number of $1$'s in the odd coordinates of $x$. Formally, viewing $x$ as a vector in $\mathbb{R}^n$ rather than $\mathbb{Z}_2^n$ and using the inner product of the vectors with real values we have:
		
		\[O(x)= o \cdot x^\mathsf{T} \text{ where } o=[1010\cdots ] \in \mathbb{R}^{n}. \]
				
		Observe that when $n \equiv 3$ or $0 \pmod 4$, since the number of odd coordinates is even, we have  $O(x) \equiv O(\bar{x}) \pmod 2$. For  $n \equiv 1$ or $2 \pmod 4$, we have $O(x) \equiv O(\bar{x})+1 \pmod 2$.
		
		In the case of $n \equiv 3$ or $0 \pmod 4$, we define $\varsigma_n(x\bar{x}) =(-1)^{O(x)}$. This is independent of which end of the edge we choose.  In the case of $n \equiv 1$ or $2 \pmod 4$, the symmetry does not hold and we must use directed edges for a function based on $O(x)$. Thus we define  $\varsigma_n(x\bar{x}) =(-1)^{O(x)}i$.
		
		It remains to verify that $\varsigma_n(C)=-1$ for every 4-cycle. This is already established for 4-cycles in $H_n$. Hence we consider a 4-cycle $C=xy\bar{y}\bar{x}$ where $x+y=\bar{y}+\bar{x}=e_j$. 
		
		Observe that $b_j(\bar{x})=j-1-b_j(x)$. Based on the parity of $j$ and the value of $n \pmod 4$ there are four cases:
		
		\begin{itemize}
			\item $j$ even, $n \equiv 3$ or $0 \pmod 4$: $\varsigma_n(xy)\varsigma_n(\bar{x}\bar{y})=-1$ and $\varsigma_n(x\bar{x})=\varsigma_n(y\bar{y})\in \{+1,-1\}$.
			
			\item $j$ even, $n \equiv 1$ or $2 \pmod 4$: $\varsigma_n(xy)\varsigma_n(\bar{x}\bar{y})=-1$ and $\varsigma_n(x\bar{x})=-\varsigma_n(\bar{y}y)\in \{i,-i\}$.
			
			\item $j$ odd, $n \equiv 3$ or $0 \pmod 4$: $\varsigma_n(xy)\varsigma_n(\bar{x}\bar{y})=1$ and $\varsigma_n(x\bar{x})=-\varsigma_n(y\bar{y})\in \{+1,-1\}$.
			
			\item $j$ odd, $n \equiv 1$ or $2 \pmod 4$: $\varsigma_n(xy)\varsigma_n(\bar{x}\bar{y})=1$ and $\varsigma_n(x\bar{x})=\varsigma_n(\bar{y}y)\in \{i,-i\}$. 
		\end{itemize}

		For the final assertions, consider an antipodal $2n$-cycle $C$ in the subgraph $H_n$ of $PC(n)$. By \Cref{coro:SignOfAntipodalCycle},
		\[
			\sigma_n(C)=+1 \quad\text{if } n\equiv0,1\pmod4,
		\]
		and
		\[
			\sigma_n(C)=-1 \quad\text{if } n\equiv2,3\pmod4.
		\]

		\begin{figure}[H]
			\centering
			
			\begin{subfigure}[b]{0.48\textwidth}
				\centering
				\begin{tikzpicture}[
					baseline={(current bounding box.center)},
					scale=0.9, 
					every node/.style={font=\fontfamily{cmr}\selectfont\footnotesize}, 
					vertex/.style={circle, draw, fill=white, inner sep=1.8pt, minimum size=5.0pt},
					edge label/.style={color=purple, font=\fontfamily{cmr}\selectfont\tiny}
					]
					
					\def\a{2.5}
					\def\b{2.8}
					
					\node[vertex] (v0) at ({ \a*cos(90) }, { \b*sin(90) }) {};
					\node[vertex] (v1) at ({ \a*cos(54) }, { \b*sin(54) }) {};
					\node[vertex] (v2) at ({ \a*cos(18) }, { \b*sin(18) }) {};
					\node[vertex] (v3) at ({ \a*cos(-18) }, { \b*sin(-18) }) {};
					\node[vertex] (v4) at ({ \a*cos(-54) }, { \b*sin(-54) }) {};
					\node[vertex] (v5) at ({ \a*cos(-90) }, { \b*sin(-90) }) {};
					\node[vertex] (v6) at ({ \a*cos(-126) }, { \b*sin(-126) }) {};
					\node[vertex] (v7) at ({ \a*cos(-162) }, { \b*sin(-162) }) {};
					\node[vertex] (v8) at ({ \a*cos(162) }, { \b*sin(162) }) {};
					\node[vertex] (v9) at ({ \a*cos(126) }, { \b*sin(126) }) {};
					
					\begin{pgfonlayer}{background}
						\draw[draw=violet, dash pattern=on 3pt off 1.5pt, line width=0.8pt] (v0) -- (v5); 
						\draw[draw=violet, dash pattern=on 3pt off 1.5pt, line width=0.8pt] (v1) -- (v6); 
						\draw[draw=violet, dash pattern=on 3pt off 1.5pt, line width=0.8pt] (v2) -- (v7); 
						\draw[draw=violet, dash pattern=on 3pt off 1.5pt, line width=0.8pt] (v3) -- (v8); 
						\draw[draw=violet, dash pattern=on 3pt off 1.5pt, line width=0.8pt] (v4) -- (v9); 
					\end{pgfonlayer}
					
					\draw[blue, thick] (v0) to [bend left=16] node[midway, sloped, above, edge label] {$e_1$} (v1);
					\draw[blue, thick] (v1) to [bend left=16] node[midway, sloped, above, edge label] {$e_2$} (v2);
					\draw[blue, thick] (v8) to [bend left=16] node[midway, sloped, above, edge label] {$e_{n-1}$} (v9);
					\draw[blue, thick] (v9) to [bend left=16] node[midway, sloped, above, edge label] {$e_n$} (v0);
					
					\draw[blue, thick] (v5) to [bend left=16] node[midway, sloped, below, edge label] {$e_1$} (v6);
					\draw[blue, thick] (v6) to [bend left=16] node[midway, sloped, below, edge label] {$e_2$} (v7);
					\draw[blue, thick] (v3) to [bend left=16] node[midway, sloped, below, edge label] {$e_{n-1}$} (v4);
					\draw[blue, thick] (v4) to [bend left=16] node[midway, sloped, below, edge label] {$e_n$} (v5);
					
					\draw[blue, loosely dotted, line width=1.3pt] (v2) to [bend left=16] (v3);
					\draw[blue, loosely dotted, line width=1.3pt] (v7) to [bend left=16] (v8);
				\end{tikzpicture}
				\caption{Antipodal cycle in $PC(n)$;}
				\label{subfig:CoordinateLayout}
			\end{subfigure}
			\hfill
			\begin{subfigure}[b]{0.48\textwidth}
				\centering
				\begin{tikzpicture}[
					baseline={(current bounding box.center)}, 
					scale=0.9, 
					every node/.style={font=\fontfamily{cmr}\selectfont\footnotesize}, 
					vertex/.style={circle, draw, fill=white, inner sep=1.8pt, minimum size=5.0pt},
					edge label/.style={color=purple, font=\fontfamily{cmr}\selectfont\tiny}, 
					rung edge/.style={draw=violet, dash pattern=on 3pt off 1.5pt, line width=0.8pt}
					]
					
					\def\h{1.0}    
					\def\w{0.85}    
					\def\gap{1.5}  
					
					\path (0, -3.05) -- (0, 2.55);
					
					\node[vertex] (t1) at ({-3.5*\w - 0.5*\gap},  \h) {};
					\node[vertex] (t2) at ({-2.5*\w - 0.5*\gap},  \h) {};
					\node[vertex] (t3) at ({-1.5*\w - 0.5*\gap},  \h) {};
					\node[vertex] (t4) at ({-0.5*\w - 0.5*\gap},  \h) {};
					\node[vertex] (t5) at ({ 0.5*\w + 0.5*\gap},  \h) {};
					\node[vertex] (t6) at ({ 1.5*\w + 0.5*\gap},  \h) {};
					\node[vertex] (t7) at ({ 2.5*\w + 0.5*\gap},  \h) {};
					\node[vertex] (t8) at ({ 3.5*\w + 0.5*\gap},  \h) {};
					
					\node[vertex] (b1) at ({-3.5*\w - 0.5*\gap}, -\h) {};
					\node[vertex] (b2) at ({-2.5*\w - 0.5*\gap}, -\h) {};
					\node[vertex] (b3) at ({-1.5*\w - 0.5*\gap}, -\h) {};
					\node[vertex] (b4) at ({-0.5*\w - 0.5*\gap}, -\h) {};
					\node[vertex] (b5) at ({ 0.5*\w + 0.5*\gap}, -\h) {};
					\node[vertex] (b6) at ({ 1.5*\w + 0.5*\gap}, -\h) {};
					\node[vertex] (b7) at ({ 2.5*\w + 0.5*\gap}, -\h) {};
					\node[vertex] (b8) at ({ 3.5*\w + 0.5*\gap}, -\h) {};
					
					\begin{pgfonlayer}{background}
						\draw[blue, thick] (t1) -- node[pos=0.40, sloped, above, edge label] {$e_n$} (b8);
						\draw[blue, thick] (b1) -- node[pos=0.60, sloped, above, edge label] {$e_n$} (t8);
					\end{pgfonlayer}
					
					\draw[rung edge] (t1) -- (b1);
					\draw[rung edge] (t2) -- (b2);
					\draw[rung edge] (t3) -- (b3);
					\draw[rung edge] (t4) -- (b4);
					\draw[rung edge] (t5) -- (b5);
					\draw[rung edge] (t6) -- (b6);
					\draw[rung edge] (t7) -- (b7);
					\draw[rung edge] (t8) -- (b8);
					
					\draw[blue, thick] (t1) -- node[midway, above, edge label] {$e_1$} (t2);
					\draw[blue, thick] (t2) -- node[midway, above, edge label] {$e_2$} (t3);
					\draw[blue, thick] (t3) -- (t4);
					\draw[blue, loosely dotted, line width=1.3pt] (t4) -- (t5);
					\draw[blue, thick] (t5) -- (t6);
					\draw[blue, thick] (t6) -- node[midway, above, edge label] {$e_{n-2}$} (t7);
					\draw[blue, thick] (t7) -- node[midway, above, edge label] {$e_{n-1}$} (t8);
					
					\draw[blue, thick] (b1) -- node[midway, below, edge label] {$e_1$} (b2);
					\draw[blue, thick] (b2) -- node[midway, below, edge label] {$e_2$} (b3);
					\draw[blue, thick] (b3) -- (b4);
					\draw[blue, loosely dotted, line width=1.3pt] (b4) -- (b5);
					\draw[blue, thick] (b5) -- (b6);
					\draw[blue, thick] (b6) -- node[midway, below, edge label] {$e_{n-2}$} (b7);
					\draw[blue, thick] (b7) -- node[midway, below, edge label] {$e_{n-1}$} (b8);
					
				\end{tikzpicture}
				\caption{Möbius ladder presentation.}
				\label{subfig:MobiusLadderProfile}
			\end{subfigure}
			
		\end{figure}

		Observe that the subgraph of $PC(n)$ induced by $C$ is the Möbius ladder with $n-1$ steps; see \Cref{subfig:MobiusLadderProfile} and \Cref{subfig:CoordinateLayout} for two presentations of this subgraph. Suppose first that $n\equiv0$ or $1\pmod4$. By applying $(-1)$-switches, if necessary, we may assume that all edges of $C$ are positive. Such switchings do not interchange real and imaginary values, so all edges corresponding to $J$ remain of the same type. If these values are real, then the signs on consecutive $J$-edges must alternate in order that each $4$-cycle have sign $-1$. The final $4$-cycle, corresponding to $e_n$ and $J$, also requires the first and last $J$-edges to have opposite signs. Hence $n-1$ must be odd, and therefore $n\equiv0\pmod4$.

		If the values on the $J$-edges are imaginary, then, after reversing orientations if necessary, they may all be regarded as having value $i$. For the $4$-cycles corresponding to $e_j$ and $J$, $j\leq n-1$, the orientations of consecutive $J$-edges must alternate, whereas the final $4$-cycle, corresponding to $e_n$ and $J$, requires the first and last $J$-edges to have the same orientation. Hence $n-1$ must be even, and therefore $n\equiv1\pmod4$.

		Suppose now that $n\equiv2$ or $3\pmod4$. Since $\sigma_n(C)=-1$, we may switch so that all edges of $C$ are positive except one of the two edges corresponding to $e_n$. The same argument applies, except that the condition imposed by the final $4$-cycle is reversed. Thus real values on the $J$-edges require $n-1$ to be even, giving $n\equiv3\pmod4$, while imaginary values require $n-1$ to be odd, giving $n\equiv2\pmod4$.

		In each case, once the value on one $J$-edge is chosen, the $4$-cycle conditions determine the values on all the remaining $J$-edges; the two choices are negatives of one another.\end{proof}

    	The assumption $n\geq4$ ensures that $PC(n)$ is a Sidon cube-like graph. Hence, the restriction of any unitary signing of $PC(n)$ to $H_n$ is a unitary signing of $H_n$. Therefore, the two switching-equivalence classes represented above are the only ones.

    For $PC(2)$, one checks directly that there are also only two switching-equivalence classes. One of the two signings is presented in \Cref{fig:K4-unitary-signing}. In the displayed representative, every $4$-cycle contains an imaginary edge.        	
	
	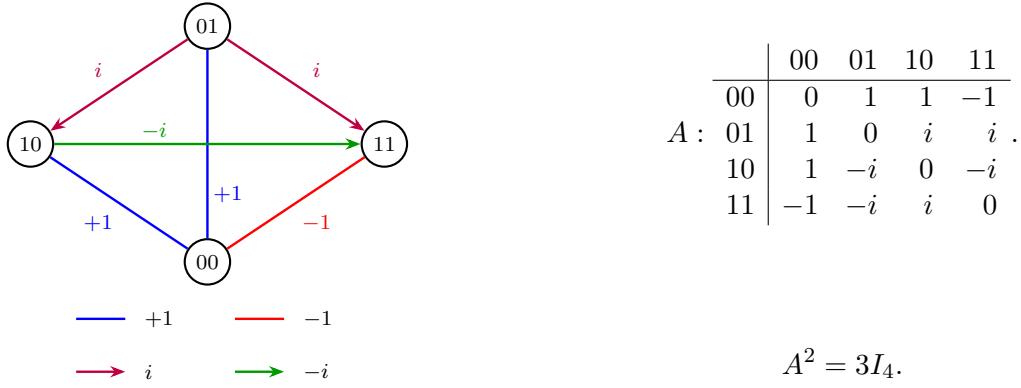
\begin{figure}[htbp]
		\centering
		
		\begin{minipage}[c]{0.48\textwidth}
			\centering
			\scriptsize
			
			\begin{tikzpicture}[
				scale=0.78,
				vertex/.style={
					circle,
					draw=black,
					fill=white,
					thick,
					minimum size=6mm,
					inner sep=0.5pt,
					font=\scriptsize
				},
				pedge/.style={
					draw=blue,
					line width=0.9pt
				},
				nedge/.style={
					draw=red,
					line width=0.9pt
				},
				iedge/.style={
					draw=purple,
					line width=0.9pt,
					-{Stealth[length=2mm,width=1.5mm]}
				},
				niedge/.style={
					draw=green!60!black,
					line width=0.9pt,
					-{Stealth[length=2mm,width=1.5mm]}
				}
				]
				
				\node[vertex] (v01) at (3,4) {\scriptsize $01$};
				\node[vertex] (v10) at (0,2) {\scriptsize $10$};
				\node[vertex] (v11) at (6,2) {\scriptsize $11$};
				\node[vertex] (v00) at (3,0) {\scriptsize $00$};
				
				\draw[pedge] (v00) -- (v01);
				\draw[pedge] (v00) -- (v10);
				\draw[nedge] (v00) -- (v11);
				
				\draw[iedge]  (v01) -- (v10);
				\draw[iedge]  (v01) -- (v11);
				\draw[niedge] (v10) -- (v11);
				
				\node[
				text=blue,
				fill=white,
				inner sep=1pt,
				font=\scriptsize
				] at (3.35,1.15) {$\scriptsize +1$};
				
				\node[
				text=blue,
				fill=white,
				inner sep=1pt,
				font=\scriptsize
				] at (1.15,0.65) {$\scriptsize +1$};
				
				\node[
				text=red,
				fill=white,
				inner sep=1pt,
				font=\scriptsize
				] at (4.85,0.65) {$\scriptsize -1$};
				
				\node[
				text=purple,
				fill=white,
				inner sep=1pt,
				font=\scriptsize
				] at (1.15,3.25) {$\scriptsize i$};
				
				\node[
				text=purple,
				fill=white,
				inner sep=1pt,
				font=\scriptsize
				] at (4.85,3.25) {$\scriptsize i$};
				
				\node[
				text=green!60!black,
				fill=white,
				inner sep=1pt,
				font=\scriptsize
				] at (2.1,2.22) {$\scriptsize -i$};
				
			\end{tikzpicture}
			
			\vspace{2mm}
			
			\begin{tikzpicture}[
				font=\scriptsize,
				pedge/.style={
					draw=blue,
					line width=0.9pt
				},
				nedge/.style={
					draw=red,
					line width=0.9pt
				},
				iedge/.style={
					draw=purple,
					line width=0.9pt,
					-{Stealth[length=2mm,width=1.5mm]}
				},
				niedge/.style={
					draw=green!60!black,
					line width=0.9pt,
					-{Stealth[length=2mm,width=1.5mm]}
				}
				]
				
				\draw[pedge] (0,1.05) -- (0.65,1.05);
				\node[right] at (0.8,1.05) {$+1$};
				
				\draw[nedge] (2.1,1.05) -- (2.75,1.05);
				\node[right] at (2.9,1.05) {$-1$};
				
				\draw[iedge] (0,0.35) -- (0.65,0.35);
				\node[right] at (0.8,0.35) {$i$};
				
				\draw[niedge] (2.1,0.35) -- (2.75,0.35);
				\node[right] at (2.9,0.35) {$-i$};
				
			\end{tikzpicture}
			
		\end{minipage}
		\hfill
		\begin{minipage}[c]{0.47\textwidth}
			\centering
			\small
			
			\[
			A:
			\begin{array}{c|rrrr}
				& 00 & 01 & 10 & 11 \\ \hline
				00 & 0  & 1  & 1  & -1 \\
				01 & 1  & 0  & i  & i  \\
				10 & 1  & -i & 0  & -i \\
				11 & -1 & -i & i  & 0
			\end{array}.
			\]
			
			\vspace{2mm}
			\hspace{1cm}
			
			\[
			A^2=3I_4.
			\]
			
		\end{minipage}
		
		\caption{A unitary signing of $K_4$.}
		\label{fig:K4-unitary-signing}
	\end{figure}

	\FloatBarrier
	
After a $(-i)$-switch at the vertex $11$, as depicted in \Cref{fig:K4-i-switch}, we get the $H_2$-subgraph $00-01-11-10$ with orthogonal signing and the other two edges are assigned imaginary values.

	\begin{figure}[H]
		\centering
		
		\begin{minipage}[c]{0.48\textwidth}
			\centering
			\scriptsize
			
			\begin{tikzpicture}[
				scale=0.78,
				vertex/.style={
					circle,
					draw=black,
					fill=white,
					thick,
					minimum size=6mm,
					inner sep=0.5pt,
					font=\scriptsize
				},
				pedge/.style={
					draw=blue,
					line width=0.9pt
				},
				nedge/.style={
					draw=red,
					line width=0.9pt
				},
				iedge/.style={
					draw=purple,
					line width=0.9pt,
					-{Stealth[length=2mm,width=1.5mm]}
				},
				niedge/.style={
					draw=green!60!black,
					line width=0.9pt,
					-{Stealth[length=2mm,width=1.5mm]}
				}
				]
				
				\node[vertex] (v01) at (3,4) {$01$};
				\node[vertex] (v10) at (0,2) {$10$};
				\node[vertex] (v11) at (6,2) {$11$};
				\node[vertex] (v00) at (3,0) {$00$};

				\draw[pedge] (v00) -- (v01);
				\draw[pedge] (v00) -- (v10);
				\draw[iedge] (v01) -- (v10);

				\draw[niedge] (v11) -- (v00);
				\draw[pedge]  (v01) -- (v11);
				\draw[nedge]  (v10) -- (v11);
				
		
				\node[
				text=blue,
				fill=white,
				inner sep=1pt,
				font=\scriptsize
				] at (3.35,1.15) {$+1$};
				
				\node[
				text=blue,
				fill=white,
				inner sep=1pt,
				font=\scriptsize
				] at (1.15,0.65) {$+1$};
				
				\node[
				text=purple,
				fill=white,
				inner sep=1pt,
				font=\scriptsize
				] at (1.15,3.25) {$i$};
				
				\node[
				text=blue,
				fill=white,
				inner sep=1pt,
				font=\scriptsize
				] at (4.85,3.25) {$+1$};
				
				\node[
				text=red,
				fill=white,
				inner sep=1pt,
				font=\scriptsize
				] at (2.1,2.22) {$-1$};
				
				\node[
				text=green!60!black,
				fill=white,
				inner sep=1pt,
				font=\scriptsize
				] at (4.85,0.65) {$-i$};
				
			\end{tikzpicture}

		\end{minipage}
		\hfill
		%
		\begin{minipage}[c]{0.47\textwidth}
			\centering
			\small
			
			\[
			A'=
			\begin{array}{c|rrrr}
				& 00 & 01 & 10 & 11 \\ \hline
				00 & 0  & 1  & 1  & i  \\
				01 & 1  & 0  & i  & 1  \\
				10 & 1  & -i & 0  & -1 \\
				11 & -i & 1  & -1 & 0
			\end{array}.
			\]

			\[
			(A')^2=3I_4.
			\]
			
		\end{minipage}
		
		\caption{The unitary signing of $K_4$ obtained by applying a
			$(-i)$-switch at the vertex $11$.}
		\label{fig:K4-i-switch}
	\end{figure}
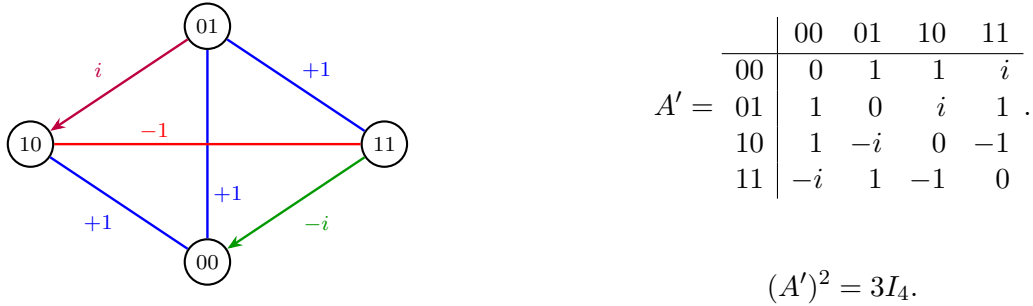
	
	For $PC(3)$, the situation is different. A direct enumeration gives $24$ switching-equivalence classes of unitary signings. Two of these classes restrict to a unitary signing on every spanning copy of $H_3$ obtained from a basis contained in $S$, whereas the remaining $22$ restrict to no such copy. Six of the $24$ classes contain orthogonal representatives. Figure~\ref{fig:PC(3)} shows a unitary signing whose restriction to any copy of $H_3$, whether basis-generated or not, is not unitary.
	
	\begin{figure}[ht]
		\centering
		\begin{minipage}[c]{0.5\textwidth}
			\centering
			\begin{tikzpicture}[
				scale=0.82,
				vertex/.style={circle,draw=black,fill=white,thick,
					minimum size=2mm,inner sep=1pt},
				pedge/.style={draw=blue,line width=1pt},
				nedge/.style={draw=red,line width=1pt},
				iarc/.style={draw=purple,line width=1pt,->,>=stealth}
				]
				\node[vertex] (v000) at (0,3.5)   {\tiny $000$};
				\node[vertex] (v011) at (2.2,3.5) {\tiny $011$};
				\node[vertex] (v101) at (4.4,3.5) {\tiny $101$};
				\node[vertex] (v110) at (6.6,3.5) {\tiny $110$};
				\node[vertex] (v001) at (0,0)   {\tiny $001$};
				\node[vertex] (v010) at (2.2,0) {\tiny $010$};
				\node[vertex] (v100) at (4.4,0) {\tiny $100$};
				\node[vertex] (v111) at (6.6,0) {\tiny $111$};
				
				\draw[pedge] (v000) -- (v001);
				\draw[pedge] (v000) -- (v010);
				\draw[pedge] (v000) -- (v100);
				\draw[pedge] (v000) -- (v111);
				\draw[pedge] (v011) -- (v001);
				\draw[nedge] (v011) -- (v100);
				\draw[pedge] (v101) -- (v001);
				\draw[nedge] (v101) -- (v010);
				\draw[pedge] (v101) -- (v100);
				\draw[nedge] (v101) -- (v111);
				\draw[pedge] (v110) -- (v001);
				\draw[nedge] (v110) -- (v100);
				\draw[iarc] (v011) -- (v010);
				\draw[iarc] (v111) -- (v011);
				\draw[iarc] (v010) -- (v110);
				\draw[iarc] (v110) -- (v111);
				
				\draw[pedge] (0.2,-0.8) -- (0.9,-0.8);
				\node[right] at (0.95,-0.8) {\small $+1$};
				\draw[nedge] (2.5,-0.8) -- (3.2,-0.8);
				\node[right] at (3.25,-0.8) {\small $-1$};
				\draw[iarc] (4.8,-0.8) -- (5.5,-0.8);
				\node[right] at (5.55,-0.8) {\small $i$};
			\end{tikzpicture}
		\end{minipage}
		\hfill
		\begin{minipage}[c]{0.43\textwidth}
			\centering
			\footnotesize
			\[
			A=\left(
			\begin{array}{rrrr|rrrr}
				0&0&0&0 & 1& 1& 1& 1\\
				0&0&0&0 & 1& i&-1&-i\\
				0&0&0&0 & 1&-1& 1&-1\\
				0&0&0&0 & 1&-i&-1& i\\ \hline
				1& 1& 1& 1 & 0&0&0&0\\
				1&-i&-1& i & 0&0&0&0\\
				1&-1& 1&-1 & 0&0&0&0\\
				1& i&-1&-i & 0&0&0&0
			\end{array}\right)
			\]
			\normalsize
			\[ A^2=4I_8. \]
		\end{minipage}
		\caption{A unitary signing of $PC(3)$ whose restriction to no
			spanning copy of $H_3$ is unitary.}
		\label{fig:PC(3)}
	\end{figure}

	Let $S$ be Sidon and suppose that $Q_S$ admits a unitary signing $\sigma$. After switching, assume that the restriction of $\sigma$ to $H_n$ is $\sigma_n$. For $s\in S\setminus\{e_1,\ldots,e_n\}$, define $\sigma^{s}$ to be a signing by
	
	\[ \sigma^{s}(xy)=
	\begin{cases}
		\sigma(xy) & \text{ if } x+y\neq s,\\
		-\sigma(xy) & \text{ if } x+y=s.
	\end{cases}
	\]
	
	Since $S$ is Sidon, every $4$-cycle containing an edge corresponding to $s$ contains exactly one other edge corresponding to $s$. Thus the sign of every $4$-cycle remains $-1$ in $\sigma^s$. For $s\in S\setminus\{e_1,\ldots,e_n\}$, consider the dependency relation
	\[
	s=e_{j_1}+\cdots+e_{j_\ell}.
	\]
	Each cycle $C$ in $Q_S$ corresponding to this relation satisfies $\sigma^s(C)=-\sigma(C)$. Hence, $\sigma$ and $\sigma^s$ are not switching equivalent.
	
	By the preceding projective-cube analysis, once the restriction to $H_n$ is fixed as $\sigma_n$, the values on the edges corresponding to each $s\in S\setminus\{e_1,\ldots,e_n\}$ are determined up to simultaneous negation. These choices may be made independently, since negating all edges corresponding to any $s$ preserves the sign of every $4$-cycle and hence, by \Cref{prop:A2-4Cycles}, the resulting signing is again unitary. Therefore, there are exactly $2^{|S|-n}$ such choices, and the preceding cycles show that they belong to distinct switching-equivalence classes.

	\begin{corollary}\label{coro:NumberOfSignatres}
		If $S$ is a Sidon subset of $\mathbb{Z}_2^n$ generating it, and $Q_S=(\mathbb{Z}_2^n,S)$ admits a unitary signing, then it has precisely $2^{|S|-n}$ switching-equivalence classes of unitary signings.
	\end{corollary}

	For a cube-like graph $Q_S=(\mathbb{Z}_2^n,S)$, a closed walk $W$ is said to be \emph{$S$-simple} if each element of $S$ corresponds to at most one edge of $W$. The preceding proof implies the following.
	
	\begin{lemma}\label{lem:S-simple-Real}
		Let $S\subseteq \mathbb{Z}_2^n$ be Sidon. If the cube-like graph $Q_S=(\mathbb{Z}_2^n, S)$ admits a unitary signing $\sigma$ whose restriction to $H_n$ is $\sigma_n$, then an $S$-simple closed walk $W$ is
		\begin{itemize}
			\item real if $|W|\equiv 0$ or $1 \pmod{4}$,
			\item imaginary if $|W|\equiv 2$ or $3 \pmod{4}$.
		\end{itemize}
	\end{lemma} 
	
	The proof of lemma is similar to that of \Cref{lem:AntipodalCycle}; we show that a closed walk $\myvec{W}$ can be reversed to $\mycev{W}$ by taking symmetric difference with $\binom{|W|}{2}$ two-label 4-cycles.  
	
	This leads to some forbidden structures in $Q_S$. Let $u$ and $v$ be two vertices of $Q_S$ joined by three walks $P_k$, $P_{\ell}$, and $P_r$ of lengths $k$, $\ell$, and $r$, respectively, such that each element of $S$ labels at most one edge among the three walks. The three closed walks
	\[
	C_{k\ell}=\myvec{P_k}\mycev{P_{\ell}},\qquad
	C_{\ell r}=\myvec{P_{\ell}}\mycev{P_r},\qquad
	C_{rk}=\myvec{P_r}\mycev{P_k}
	\]
	are $S$-simple and have respective lengths $k+\ell$, $\ell+r$, and $r+k$. In the product of their signs, each edge contribution appears once in each direction, and hence
	\[
	\sigma(C_{k\ell})\sigma(C_{\ell r})\sigma(C_{rk})=1.
	\]
	Thus at least one of these three signs, say $\sigma(C_{k\ell})$, is real, and the other two are either both real or both imaginary. By \Cref{lem:S-simple-Real}, we have $k+\ell \equiv 0$ or $1 \pmod{4}$. Furthermore, $k+r$ and $\ell+r$ are either both $0$ or $1 \pmod{4}$, or they are both $2$ or $3 \pmod{4}$.
	Using the symmetry between $k$ and $\ell$, this leads to the following possible triples for $(k,\ell,r)$ modulo $4$: $(0,0,x)$, $(0,1,2x)$, $(2,2,x)$, $(2,3,2x)$, where $x$ can be any element of $\{0,1,2,3\}$. These are precisely the triples in which at least two entries are even.
	
	The forbidden cases are summarized in the following theorem.
	
	\begin{theorem}\label{thm:ForbiddenTheta}
		Let $Q_S=(\mathbb{Z}_2^n,S)$ be a Sidon cube-like graph containing a $\Theta$-subgraph whose edges have distinct labels from $S$. If at least two of the three paths of the $\Theta$-subgraph have odd length, then $Q_S$ admits no unitary signing.
	\end{theorem} 
	
	More generally, suppose that $S$ contains three pairwise disjoint subsets $A,B,C$ whose element sums are equal to the same value $u$. Ordering the elements of each subset gives three walks $P_A,P_B,P_C$ from $0$ to $u$, whose edge labels are $A,B,C$, respectively. These walks need not be internally vertex disjoint. Their pairwise concatenations $P_A\overleftarrow{P_B}$, $P_B\overleftarrow{P_C}$, and $P_C\overleftarrow{P_A}$ are $S$-simple closed walks, so the preceding argument applies. Hence, if at least two of $A,B,C$ have odd cardinality, then $Q_S$ admits no unitary signing. The case in which one of the three sets is empty is included, with the corresponding walk taken to be the trivial walk at $0$.

	In particular, if $S$ contains two disjoint odd dependency relations, then $Q_S$ admits no unitary signing.
	
	We call a $4$-cycle in a cube-like graph a \emph{two-label $4$-cycle} if its edge labels, in cyclic order, are $s,t,s,t$ for distinct labels $s$ and $t$. We also record the following related result for Cartesian products of cycles.
	
	\begin{theorem}\label{thm:CartesianOddCycles}
		Let $k$ and $\ell$ be positive integers with $k,\ell \geq 3$. The Cartesian product
		\( C_{k}\square C_{\ell} \)
		admits a complex signing in which every two-label $4$-cycle has sign $-1$ if and only if $k\ell$ is even.
	\end{theorem}
	
	\begin{proof}
		For sufficiency, if $k\ell$ is even, then at least one of $k$ and $\ell$, say $k$, is even. Then in \( C_{k}\square C_{\ell} \) assign $+1$ to each edge of every copy of $C_{k}$. For the copies of $C_{\ell}$, alternately assign $+1$ or $-1$ to all edges. This is possible because there are an even number of such copies.  
		
		For necessity, assume $k\ell$ is odd and, toward a contradiction, suppose that \(\sigma\) is a complex signing on \(C_{k}\square C_{\ell} \) in which every two-label $4$-cycle has sign $-1$. Observe that \( C_{k}\square C_{\ell} \) admits a natural embedding on the torus whose faces are precisely these two-label $4$-cycles. Since the torus is orientable, we may orient all faces consistently so that each edge belongs to two faces, with opposite orientations. Thus the product of the two signs on each edge with respect to its orientations on the two faces it belongs to is $+1$. So the product of the signs of all faces is $+1$ as well. This contradicts the assumption that the sign of each face is $-1$ because there is an odd number of them ($k\ell$). 
	\end{proof}

Once the existence of a unitary signing is settled, one may ask when such a signing can be chosen to be orthogonal. The preceding description of $S$-simple closed walks gives a direct answer. For a minimal dependency relation
\[
D=\{e_1,\ldots,e_j,s\},
\qquad
s=e_1+\cdots+e_j,
\]
the necessity can also be seen directly from the projective-cube theorem: the subgraph obtained by keeping only the edges with labels in $D$ is a disjoint union of copies of $PC(j)$, and, since $S$ is Sidon, the restriction of a unitary signing is still unitary on each copy. The argument below treats all $S$-simple closed walks at once.

\begin{theorem}\label{thm:Unitary->Orthogonal}
	If a Sidon cube-like graph $Q_S$ admits a unitary signing, then it admits an orthogonal signing if and only if it has no $S$-simple closed walk whose length is congruent to $2$ or $3$ modulo $4$.
\end{theorem}

\begin{proof}
	Suppose first that $Q_S$ admits an orthogonal signing. Since $S$ is Sidon, the restriction of this signing to a spanning copy of $H_n$ is an orthogonal signing of $H_n$. After switching, we may assume that this restriction is $\sigma_n$. Every closed walk in an orthogonal signing is real, whereas by \Cref{lem:S-simple-Real} every $S$-simple closed walk of length congruent to $2$ or $3$ modulo $4$ is imaginary. Hence no such $S$-simple closed walk can occur.
	
	Conversely, suppose that $Q_S$ admits a unitary signing and has no $S$-simple closed walk whose length is congruent to $2$ or $3$ modulo $4$. After switching, we may assume that the restriction of the signing to a spanning copy of $H_n$ is $\sigma_n$. Let
	\[
	s\in S\setminus\{e_1,\ldots,e_n\},
	\]
	and write
	\[
	E_s=\{e_{j_1},\ldots,e_{j_k}\},
	\qquad
	s=e_{j_1}+\cdots+e_{j_k}.
	\]
	Then
	\[
	0,\ e_{j_1},\ e_{j_1}+e_{j_2},\ \ldots,\ e_{j_1}+\cdots+e_{j_k}=s,\ 0
	\]
	is an $S$-simple closed walk $C_s$ of length $k+1=|s|+1$, whose last edge is labeled $s$ and whose other edges lie in $H_n$. By assumption,
	\[
	|C_s|\equiv0\text{ or }1\pmod4,
	\]
	so \Cref{lem:S-simple-Real} implies that $C_s$ is real. All the edges of $C_s$ lying in $H_n$ have real signs under $\sigma_n$, and therefore the edge labeled $s$ is also real. By \Cref{lem:edge-type}, all edges of $Q_S$ labeled $s$ are then real. Since this holds for every $s\in S\setminus\{e_1,\ldots,e_n\}$, the entire signing is real, and hence it is an orthogonal signing.
\end{proof}

	\section{$\Theta$-property}
	
    We next show that the $\Theta$-property is the key condition governing the existence of the signings considered here. 
    
    We first address the following general question: given a graph and a partition of its closed walks into four sets labeled $\{i, i^2, i^3, i^4\}$, does there exist a complex signing of the edges for which the associated sign of any closed walk is the given one? 
    In this section, a closed walk is regarded up to cyclic rotation of its vertex sequence; in particular, no starting vertex is distinguished. Thus, if $W=v_0v_1\cdots v_{k-1}v_0$, every cyclic rotation of this sequence represents the same closed walk. Trivial walks are also allowed, and their sign under a complex signing is $1$.
    The answer, together with the details, is given in the following theorem which extends the answer to a similar question in \cite{NSZ21}. 
	
	\begin{theorem}
		A mapping $\eta$ of the closed walks of a graph $G$ to $\{i,i^2,i^3,i^4\}$ is obtained from a complex signing $\sigma$ on $G$ if and only if, for any three walks $\myvec{P}$, $\myvec{Q}$, and $\myvec{R}$ with the same endpoints,
		\[
		\eta(\myvec{P}\mycev{Q})\eta(\myvec{Q}\mycev{R})\eta(\myvec{R}\mycev{P})=1.
		\]
	\end{theorem}

	\begin{proof}
		Our proof follows the same ideas as in \cite{NSZ21} (Theorem 10). Necessity is immediate. For sufficiency, first observe that the stated condition implies the usual rule for reversing a closed walk. Let $E_v$ denote the trivial walk at a vertex $v$. Taking $\myvec{P}=\myvec{Q}=\myvec{R}=E_v$ gives $\eta(E_v)^3=1$, and hence $\eta(E_v)=1$. More generally, taking $\myvec{P}=\myvec{Q}=\myvec{R}$ gives
		\[
		\eta(\myvec{P}\mycev{P})=1.
		\]
		If $\myvec{W}$ is a closed walk based at $v$, taking $\myvec{P}=\myvec{W}$ and $\myvec{Q}=\myvec{R}=E_v$ gives
		\[
		\eta(\myvec{W})\eta(\mycev{W})=1,
		\]
		and therefore $\eta(\mycev{W})=\overline{\eta(\myvec{W})}$. Similarly, if $\myvec{A}$ and $\myvec{B}$ are closed walks based at the same vertex, applying the condition with $\myvec{P}=\myvec{A}$, $\myvec{Q}=\mycev{B}$, and $\myvec{R}=E_v$ gives
		\[
		\eta(\myvec{A}\myvec{B})=\eta(\myvec{A})\eta(\myvec{B}).
		\]
		Since closed walks are considered up to cyclic rotation, these observations also show that inserting or deleting a backtrack $xyx$ does not change the value of $\eta$.
		
		We may now work with connected graphs. Take a spanning tree $T$ and set $\sigma(e)=1$ for each edge $e$ of $T$. For any other directed edge $xy$, let $\myvec{C}_{xy}$ be the unique cycle in $T+xy$, directed so that it traverses $xy$ from $x$ to $y$, and define
		\(
		\sigma(\myvec{xy})=\eta(\myvec{C}_{xy}).
		\)
		The reversal rule shows that this gives a well-defined complex signing. It remains to show that $\eta(\myvec{W})=\sigma(\myvec{W})$ for every closed walk $W$. If $W$ is contained in $T$, then successive deletion of backtracks reduces it to a trivial walk, so $\eta(W)=1=\sigma(W)$. Otherwise, after a cyclic rotation, write $W=xyP$, where $xy$ is an edge outside $T$, and let $T_{xy}$ be the $x$--$y$ path in $T$. Let $W_{xy}=T_{xy}P$. Applying the stated condition to the three $x$--$y$ walks $xy$, $T_{xy}$, and $\mycev{P}$ gives
		\[
		\eta(\myvec{C}_{xy})\eta(\myvec{W}_{xy})\overline{\eta(\myvec{W})}=1 \quad
		\Rightarrow \quad
		\eta(\myvec{W})=\sigma(\myvec{xy})\eta(\myvec{W}_{xy}).
		\]
		Since $W_{xy}$ contains one fewer edge outside $T$, by induction on the number of such edges we conclude that $\eta(\myvec{W})=\sigma(\myvec{W})$ for every closed walk $W$.
	\end{proof}

\subsection{$\Theta$-property and unitary signings}

The preceding obstructions can be expressed uniformly as a property of the generating set $S$.

\begin{definition}
		A subset $S$ of $\mathbb{Z}_2^n$ is said to have the $\Theta$-property if for any three pairwise disjoint subsets of $S$ whose sums are equal at least two are of even size.
	\end{definition}
	
	One of the three disjoint subsets is allowed to be empty. The sum over the empty set is understood to be $0$. Thus the condition implies that $S$ cannot have two disjoint subsets of odd size for each of which the sum of the elements is $0$. A useful reformulation of the $\Theta$-property on a subset of $\mathbb{Z}_2^n$ is as follows.

\begin{lemma}\label{lem:Theta-q-linear}
	Let $S\subseteq\mathbb Z_2^n$, and let
	\(
	\mathcal D
	=
	\left\{D\subseteq S:\sum_{d\in D}d=0\right\}
	\)
	be the space of dependency relations of $S$, where addition is symmetric difference. Define
	\(
	q(D)=\binom{|D|}{2}\pmod2.
	\)
	Then the following statements are equivalent:
	\begin{enumerate}
		\item[(1)] $S$ has the $\Theta$-property;
		\item[(2)] $q$ is a linear function on $\mathcal D$;
		\item[(3)] for every $D_1,D_2\in\mathcal D$, we have $|D_1\cap D_2|\equiv |D_1||D_2|\pmod2$.
		\end{enumerate}
	
	Equivalently, two dependency relations have odd intersection if and only if both have odd size.
\end{lemma}

\begin{proof}
	We first prove the equivalence of (1) and (3). Given $D_1,D_2\in\mathcal D$ let
	\(
	A=D_1\cap D_2,
	B=D_1\setminus D_2,
	C=D_2\setminus D_1.
	\)
	Observe that $A,B,C$ are pairwise disjoint. Since $D_1$ and $D_2$ are dependency relations,
	\(
	\sum_{a\in A}a
	=
	\sum_{b\in B}b
	=
	\sum_{c\in C}c.
	\)
	
	Let
	\(
	\alpha=|A|\pmod2,
	\beta=|B|\pmod2,
	\gamma=|C|\pmod2.
	\)
	Then
	\(
	|D_1|\equiv\alpha+\beta,
	|D_2|\equiv\alpha+\gamma
	\pmod2.
	\)
	The condition that at most one of $|A|,|B|,|C|$ is odd is equivalent to
	\(
	\alpha=(\alpha+\beta)(\alpha+\gamma),
	\)
	which is precisely
	\(
	|D_1\cap D_2|\equiv |D_1||D_2|\pmod2.
	\)
	This proves that the $\Theta$-property implies (3).

	 Conversely, if $A,B,C\subseteq S$ are pairwise disjoint and have the same sum, then
	\(
	D_1=A\cup B,
	D_2=A\cup C
	\)
	belong to $\mathcal D$. Applying (3) to $D_1,D_2$ gives the same parity identity, and hence at most one of $|A|,|B|,|C|$ is odd. Thus (3) implies the $\Theta$-property.\\
	
	It remains to prove the equivalence of (2) and (3). Let $x_D$ denote the incidence vector of $D$. Then we may view $q(D)$ as 
	\(
	q(D)=\sum_{k<l}(x_D)_k(x_D)_l\pmod2.
	\)
	
	By applying this expansion to the two sides we may verify that 
	\[
	q(D_1\oplus D_2)+q(D_1)+q(D_2)
	\equiv
	|D_1||D_2|+|D_1\cap D_2|
	\pmod2.
	\]
	Therefore $q(D_1\oplus D_2)=q(D_1)+q(D_2)$ for every $D_1,D_2\in\mathcal D$ if and only if (3) holds. 
\end{proof}

We are now ready to state and prove one of the main conclusions of this work.

\begin{theorem}\label{thm:ThetaSufficient}
Let $S\subseteq\mathbb Z_2^n$ generate $\mathbb Z_2^n$ and $0\notin S$. If $S$ has the $\Theta$-property, then $Q_S$ admits a unitary signing. Moreover, the signing may be chosen so that every two-label $4$-cycle, 
(edge labels $a,b,a,b$ for some $a,b\in S$, $a\neq b$)
has sign $-1$.
\end{theorem}

\begin{proof}
	After applying an automorphism of $\mathbb Z_2^n$, we may assume that $S$ contains the standard basis, and we start with the signing $\sigma_n$ on the subgraph $H_n$. We write $x=(x_1,x_2,\dots,x_n)$ and $s=(s_1,s_2,\dots,s_n)$. Moreover, for $s\in S\setminus\{e_1,\ldots,e_n\}$, let $E_s$ be the subset of the standard basis whose elements sum to $s$, and set \(D_s=\{s\}\cup E_s\). Then $D_s$ is a dependency relation of size $|s|+1$.

	Recall that
	$\sigma_n(x,x+e_j)=(-1)^{b_j(x)}$, where $b_j(x)=\sum_{\ell<j}x_\ell$.
	It remains to define $\sigma$ on the edges corresponding to
	\(
	s\in S-\{e_1,\dots,e_n\}.
	\) We give the exact formula for the extension. For $s\in S-\{e_1,\dots,e_n\}$ and $x\in\mathbb Z_2^n$, define
	\(
	f_s(x)
	=
	\sum_{j=1}^n x_j\bigl(1+b_j(s)\bigr)
	\pmod2.
	\)
	Equivalently, by replacing $b_j(s)$ with its defining formula,
	\(
	f_s(x)
	\equiv
	|x|+
	\sum_{\ell<j}s_\ell x_j
	\pmod2.
	\)
We define
	\[
	\sigma(x,x+s)=
	\begin{cases}
		(-1)^{f_s(x)},
		& |s|\equiv0\text{ or }3\pmod4,\\[2mm]
		i(-1)^{f_s(x)},
		& |s|\equiv1\text{ or }2\pmod4.
	\end{cases}
	\]
	
	Our first task is to show that this extension of $\sigma$ is well-defined. Observe that $f_s$ is linear as a function of $x$, i.e., $f_s(x+y)=f_s(x)+f_s(y)$ where all operations are modulo $2$. In particular, we have
	\(
	f_s(x+s)-f_s(x)=f_s(s)\pmod2.
	\)
    Observe that
	\[
	f_s(s)
	\equiv
	\sum_{j=1}^n s_j(1+b_j(s))\equiv
	|s|+\sum_{\ell<j} s_{\ell}s_j\equiv
	|s|+\binom{|s|}{2}= \frac{|s|(|s|+1)}2
	\pmod2.
    \]

	Consequently,
	\(
	f_s(s)\equiv
	\begin{cases}
		0,& |s| \equiv0\text{ or }3\pmod4,\\
		1,& |s| \equiv1\text{ or }2\pmod4.
	\end{cases}
	\)
	\\
	
	If $|s|\equiv0$ or $3\pmod4$, it follows that
	\(
	\sigma(x,x+s)=\sigma(x+s,x)\in\{+1,-1\}.
	\)
	Thus in this case the formula gives a well-defined value on the unoriented edge.
	
	If $|s|\equiv1$ or $2\pmod4$, then
	\(
	\sigma(x+s,x)
	=
	-\sigma(x,x+s)
	=
	\overline{\sigma(x,x+s)}.
	\)
	Thus in this case the edge is treated as an arc and the formula gives a well-defined value in $\{i,-i\}$. 
	
	A reminder that:
	\(
	\sigma(0,s)=
	\begin{cases}
		1,&|s|\equiv0\text{ or }3\pmod4,\\
		i,&|s|\equiv1\text{ or }2\pmod4.
	\end{cases}
	\)\\
	
	We claim that under this signing every two-label $4$-cycle, namely every $4$-cycle whose edge labels alternate between two distinct elements of $S$, has sign $-1$. The two-label $4$-cycles whose two labels are standard basis elements already have sign $-1$ under $\sigma_n$. Thus there are two cases left to consider.
	
	We first consider a $4$-cycle whose edges correspond to $e_j$ and $s$, where
	\(
	s\in S-\{e_1,\dots,e_n\}.
	\)
	Assume its vertices are
	\(
	x, x+e_j, x+e_j+s, x+s, x.
	\)
	For the two edges corresponding to $e_j$, observe that
	\(
	b_j(x+s)-b_j(x)=b_j(s)\pmod2.
	\)
	Thus their contribution to the sign of the $4$-cycle is
	\(
	(-1)^{b_j(s)}.
	\)
	On the other hand,
	\(
	f_s(x+e_j)-f_s(x)
	=
	1+b_j(s)
	\pmod2.
	\)
	Thus, the contribution of the two edges corresponding to $s$ is therefore
	\(
	(-1)^{1+b_j(s)}.
	\)

	In the complex case the two $s$-edges are traversed in opposite directions, so the extra factors $i$ and $-i$ cancel. Hence, the sign of the $4$-cycle is
	\(
	(-1)^{b_j(s)}(-1)^{1+b_j(s)}=-1.
	\)
	
	Finally, consider a $4$-cycle whose edges correspond to two distinct elements
	\(
	s,t\in S-\{e_1,\dots,e_n\},
	\)
with vertices
	\(
	x, x+s, x+s+t, x+t, x.
	\)
	The extra factors $1$ or $i$ in the definition of $\sigma$ cancel between the two edges having the same label, taking the conjugate when an edge is traversed in the reverse direction. Thus the sign of this cycle is
	\(
	(-1)^{f_s(t)+f_t(s)}.
	\)
	It remains to show that
	\(
	f_s(t)+f_t(s)\equiv1\pmod2.
	\) Since $D_s,D_t\in\mathcal D$, \Cref{lem:Theta-q-linear} gives
        \[
        |E_s\cap E_t|=|D_s\cap D_t|
        \equiv |D_s||D_t|
        \equiv (|s|+1)(|t|+1)
        \pmod2.
        \]
        
        From the definition of $f_s$,
        \[
        \begin{aligned}
        f_s(t)+f_t(s)
        &\equiv |s|+|t|
        +\sum_{\ell<j}s_\ell t_j
        +\sum_{\ell<j}t_\ell s_j\\
        &\equiv |s|+|t|+|s||t|-|E_s\cap E_t|\\
        &\equiv 1\pmod2.
        \end{aligned}
        \]

    Hence, every two-label $4$-cycle has sign $-1$.

    It remains to verify that the signing is unitary. Let $M$ be the Hermitian adjacency matrix associated with $\sigma$, and let $x\neq y$. Let $u=x+y\neq0$. A common neighbor of $x$ and $y$ has the form $x+a=y+b$ for some $a,b\in S$ satisfying $a+b=u$. Since $u\neq0$, we have $a\neq b$. Pair the ordered pair $(a,b)$ with $(b,a)$. The corresponding two common neighbors are $x+a$ and $x+b$, and the two length-$2$ walks together form the two-label $4$-cycle
    \(
        x, x+a, x+a+b, x+b, x.
    \)
    Since this cycle has sign $-1$, the two corresponding summands in the $(x,y)$-entry of $M\overline{M}^{\,T}$ cancel. Since $M$ is Hermitian, this says $(M^2)_{xy}=0$ for all $x\neq y$. On the diagonal, every row has exactly \(|S|\) nonzero entries. Since \(M\) is Hermitian, \(M_{yx}=\overline{M}_{xy}\), and hence $ M_{xy}M_{yx}=M_{xy}\overline{M}_{xy}=1 $     
    for every edge \(xy\). Therefore, \((M^2)_{xx}=|S|\).
    Hence,
    $M^2=|S|I$, which means $\sigma$ is a unitary signing of $Q_S$. 
\end{proof}

\begin{remark}
When $S$ is not Sidon, $Q_S$ may also contain $4$-cycles with four distinct edge labels $a,b,c,d$ satisfying $a+b+c+d=0$. The preceding construction does not prescribe the signs of these cycles, nor is such a prescription needed for the signing to be unitary. The paths of length two between two vertices $x$ and $y$ can be paired up with edge labels $(a,b)$ and $(b,a)$, so that their contributions cancel in the corresponding matrix entry.
\end{remark}

\begin{theorem}\label{thm:SidonCharacterization}
Let $S\subseteq\mathbb Z_2^n$ be a Sidon set generating $\mathbb Z_2^n$. Then $Q_S$ admits a unitary signing if and only if $S$ has the $\Theta$-property. It admits an orthogonal signing if and only if, furthermore, $S$ has no minimal dependency relation whose size is congruent to $2$ or $3$ modulo $4$.
\end{theorem}

\begin{proof}
Necessity of the $\Theta$-property for Sidon sets was proved above, while sufficiency follows from \Cref{thm:ThetaSufficient}. The statement about orthogonal signings is a direct consequence of \Cref{thm:Unitary->Orthogonal}.
\end{proof}

\begin{corollary}\label{coro:QuadraticObstruction}
	Let $S\subseteq\mathbb Z_2^n$ be a Sidon set generating
	$\mathbb Z_2^n$. If $Q_S$ admits no unitary signing, then it contains
	a subgraph $H$ on at most
	\(
	(n+1)^2+2n+3
	\)
	vertices for which there is no complex signing in which every
	two-label $4$-cycle of $H$ has sign $-1$.
\end{corollary}

\begin{proof}
	After applying an automorphism of $\mathbb Z_2^n$, we may assume that
	$S$ contains the standard basis
	\(
	B=\{e_1,\ldots,e_n\}.
	\)
	For each $s\in S\setminus B$, let
	\(
	D_s=\{s\}\cup E_s,
	\)
	where $E_s\subseteq B$ is the unique subset whose elements sum to $s$.
	The sets $D_s$, $s\in S\setminus B$, form a basis of the dependency
	space $\mathcal D$ of $S$.
	
	Since $Q_S$ admits no unitary signing, \Cref{thm:SidonCharacterization}
	implies that $S$ does not have the $\Theta$-property. By
	\Cref{lem:Theta-q-linear}, there are two fundamental
	dependencies $D_s$ and $D_t$ such that
	\[
	|D_s\cap D_t|+|D_s||D_t|\equiv1\pmod2.
	\]
	
	Choose closed walks $W_s$ and $W_t$, based at $0$, whose edge-label
	sequences list the elements of $D_s$ and $D_t$, respectively, with
	each element occurring once. Such walks are closed because $D_s$ and
	$D_t$ are dependency relations. Their concatenations
	$W_sW_t$ and $W_tW_s$ are therefore also closed walks based at $0$.
	
	Starting with the edge-label sequence of $W_sW_t$, move the labels
	of $W_s$ successively past the labels of $W_t$ by adjacent
	interchanges until the edge-label sequence of $W_tW_s$ is obtained.
	Consider one such interchange. If the two labels are equal, nothing
	changes. If the labels are distinct, say $a$ and $b$, then a segment
	$x, x+a, x+a+b$
	of the current walk is replaced by
	$x, x+b, x+a+b$.
	The union of these two length-$2$ paths is a two-label $4$-cycle, with
	cyclic edge-label sequence $a,b,a,b$.
	
	Let $H$ be the union of the initial closed walk $W_sW_t$ and all the
	two-label $4$-cycles arising from interchanges of distinct labels.
	Since
	\(
	|D_s|,|D_t|\leq n+1,
	\)
	the initial closed walk has length at most $2n+2$, and hence uses at
	most $2n+3$ vertices. Each interchange of two distinct labels adds at
	most one new vertex to the union already constructed.
	
	Every label of $W_s$ is moved past every label of $W_t$. There are
	\(
	|D_s||D_t|-|D_s\cap D_t|
	\)
	interchanges involving distinct labels, since a label common to
	$D_s$ and $D_t$ gives one interchange of equal labels. Consequently,
	\[
	|V(H)|
	\leq
	2n+3+|D_s||D_t|-|D_s\cap D_t|
	\leq
	2n+3+(n+1)^2.
	\]
	
	Suppose now that $H$ admits a complex signing in which every
	two-label $4$-cycle has sign $-1$. Consider again an interchange of
	two distinct adjacent labels $a$ and $b$. The two corresponding
	length-$2$ paths have the same endpoints and together form a
	two-label $4$-cycle of sign $-1$. It follows that their signs differ
	by a factor of $-1$. Hence, each such interchange changes the sign of
	the corresponding closed walk by a factor of $-1$.
	
	The number of interchanges of distinct labels has parity
	\[
	|D_s||D_t|-|D_s\cap D_t|
	\equiv
	|D_s||D_t|+|D_s\cap D_t|
	\equiv 1 \pmod2.
	\]
	Therefore $W_sW_t$ and $W_tW_s$ must have opposite signs:
	\(
	\sigma(W_sW_t)=-\sigma(W_tW_s).
	\)
	
	On the other hand, $W_s$ and $W_t$ are closed walks based at $0$. Thus
	\(
	\sigma(W_sW_t)
	=\sigma(W_s)\sigma(W_t)
	=\sigma(W_t)\sigma(W_s)
	=\sigma(W_tW_s),
	\)
	a contradiction. Hence, no complex signing of $H$ can make every
	two-label $4$-cycle negative.
\end{proof}

Let
$B=\{e_1,\ldots,e_n\},$
$u_j=e_1+\cdots+e_j$ for $2\leq j\leq n$, 
$U=\{u_2,\ldots,u_n\}$.
Theorem~1.2 of Alon and Zheng \cite{AZ20} assumes $S\subseteq B\cup U$ and $|S\cap U|\leq1$. Their Corollary~3.3 also treats certain cases with two elements of $U$. From the viewpoint of the $\Theta$-property, the situation for sets containing all of $B$ is particularly simple.

\begin{lemma}
	Let $U'\subseteq U$ and set $S=B\cup U'$. Then $S$ has the $\Theta$-property if and only if either $|U'|\leq1$, or
	\(
	U'=\{u_j,u_k\}, j<k,
	\)
	with $j$ even and $k$ odd. In particular, no set $B\cup U'$ with $|U'|\geq3$ has the $\Theta$-property.
\end{lemma}

\begin{proof}
	For $u_k\in U'$, the corresponding fundamental dependency is
	\[
	D_k=\{u_k,e_1,\ldots,e_k\},
	\qquad |D_k|=k+1.
	\]
	
	For $j<k$,
	\(
	|D_j\cap D_k|=j.
	\)	
	By \Cref{lem:Theta-q-linear}, the pair $D_j,D_k$ is compatible with the $\Theta$-property precisely when
	\(
	j\equiv(j+1)(k+1)\pmod2,
	\)
	which holds exactly when $j$ is even and $k$ is odd. Thus two elements $u_j,u_k$ with $j<k$ are allowed exactly in this parity pattern. Three elements $u_j,u_k,u_l$ with $j<k<l$ are impossible, since the pair $(j,k)$ would force $k$ to be odd while the pair $(k,l)$ would force $k$ to be even.
\end{proof}

\section{Extremal sets with the $\Theta$-property}

The connection that we have established between unitary signing and $\Theta$-property raises two questions: 1. What is the largest generating subset of $\mathbb Z_2^n$ with the $\Theta$-property? 2. What is the largest generating Sidon subset of $\mathbb Z_2^n$ with the $\Theta$-property?

Though we show that the tight general upper bound is $2n+1$, the distinction is interesting: without the restriction on the Sidon sets the upper bound of $2n+1$ works for almost all values of $n$ and constructions are simple. When restricted to Sidon sets, for $3\leq n\leq 9$ we get different answers. For $n\geq 10$, the upper bound of $2n+1$ is tight again, but we get a strong connection between extremal examples and binary codes.\\

Let
\(
M_\Theta(n)=\max\bigl\{|S|:S\subseteq\mathbb Z_2^n\text{ generates }\mathbb Z_2^n\text{ and has the }\Theta\text{-property}\bigr\}.
\)\\

Since an invertible linear transformation preserves both the generation and the $\Theta$-property, in studying $M_\Theta(n)$ we may equivalently require that $S$ contains the standard basis.

\subsection{The dependency space and a universal upper bound}

We continue to use the dependency space $\mathcal D$ introduced in \Cref{lem:Theta-q-linear}. The parity characterization in \Cref{lem:Theta-q-linear} immediately yields a universal upper bound.

\begin{theorem}\label{thm:upper}
	For every $n\geq 1$,
	\(
	M_\Theta(n)\leq 2n+1.
	\)
\end{theorem}

\begin{proof}
Since $S$ generates $\F_2^n$, we have $\dim\mathcal D=|S|-n$. For $D\in\mathcal D$, let $x_D\in\F_2^{|S|}$ be its incidence vector and define
\(
\widehat D=(x_D,|D|\bmod2)\in\F_2^{|S|+1}.
\)
Applying \Cref{lem:Theta-q-linear} to $D_1,D_2\in\mathcal D$, we have
\(
\widehat D_1\cdot\widehat D_2
\equiv |D_1\cap D_2|+|D_1||D_2|
\equiv0\pmod2.
\)
Thus $\widehat{\mathcal D}=\{\widehat D:D\in\mathcal D\}$ is self-orthogonal. The map $D\mapsto\widehat D$ is injective and linear, so
\(
|S|-n=\dim\widehat{\mathcal D}\leq\frac{|S|+1}{2}.
\)
Hence, $|S|\leq2n+1$.
\end{proof}

\begin{remark}\label{rem:even-dependencies}
If $S$ has the $\Theta$-property and every dependency relation of $S$ has even size, then the same argument can be applied directly to the incidence vectors $x_D$ (without the additional coordinate). In that case $\mathcal D$ itself is self-orthogonal in $\F_2^{|S|}$, and therefore $|S|\leq2n$.
\end{remark}

The bound in \Cref{thm:upper} is sharp in every dimension $n\geq3$. If one additionally requires $0\notin S$, then the only exceptional dimension is $n=4$, where the maximum is $8$ rather than $9$.

\begin{theorem}\label{thm:Theta-extremal}
	For every $n\geq3$,
	$M_\Theta(n)=2n+1.$
		More precisely, let
	$J=e_1+\cdots+e_n.$
	If $n$ is odd, then
	$S=\{e_1,\ldots,e_n\}\cup\{J+e_1,\ldots,J+e_n\}\cup\{J\}$ has the $\Theta$-property. If $n\geq6$ is even, then
	$S=\{e_1,\ldots,e_n\}\cup\{e_1+e_2,\ e_1+e_3,\ e_2+e_3,\ e_4+\cdots+e_n\}\cup\{J+e_k:4\leq k\leq n\}$ has the $\Theta$-property. In both cases, $0\notin S$ and $|S|=2n+1$.
	
	For $n=4$, the set
	$S=\{e_1,e_2,e_3,e_4\}\cup\{J+e_1,J+e_2,J+e_3,J+e_4\}\cup\{0\}$
	has the $\Theta$-property and has size $9$.
\end{theorem}

\begin{proof}
	All the displayed sets contain the standard basis and hence generate
	$\F_2^n$. Their elements are pairwise distinct in the stated ranges,
	so in each case $|S|=2n+1$.
	
	We use \Cref{lem:Theta-q-linear}. Let
	$S=\{e_1,\ldots,e_n\}\cup T.$
	For $t\in T$, let $E_t$ be the subset of the standard basis whose
	elements sum to $t$, and let
	$D_t=\{t\}\cup E_t.$
	The sets $D_t$, $t\in T$, form a basis of the dependency space
	$\mathcal D$. Since
	\[
	(D,D')
	\longmapsto
	|D\cap D'|+|D||D'|
	\pmod2
	\]
	is bilinear on $\mathcal D$, it is enough to verify that for distinct
	$t,u\in T$,
	\begin{equation}\label{eq:Theta-normalized}
		|E_t\cap E_u|
		\equiv
		(1+|t|)(1+|u|)
		\pmod2.
	\end{equation}
	
	Suppose first that $n$ is odd. Let
		$f_k=J+e_k.$
	Then
	$T=\{f_1,\ldots,f_n,J\}.$
	Each $f_k$ has even weight $n-1$, whereas $J$ has odd weight $n$.
	For $k\neq l$,
	\[
	|E_{f_k}\cap E_{f_l}|=n-2\equiv1\pmod2,
	\text{ while }
	|E_{f_k}\cap E_J|=n-1\equiv0\pmod2.
	\]
	These are exactly the parities required by
	\eqref{eq:Theta-normalized}. Hence, $S$ has the $\Theta$-property.
	
	Now suppose that $n\geq6$ is even. Let
	$a_{12}=e_1+e_2,
	a_{13}=e_1+e_3,
	a_{23}=e_2+e_3,$
	$h=e_4+\cdots+e_n,
	f_k=J+e_k$, where $4\leq k\leq n.$
	Then
	$T=
	\{a_{12},a_{13},a_{23},h\}
	\cup
	\{f_k:4\leq k\leq n\}.$
	The vectors $a_{12},a_{13},a_{23}$ have even weight, while
	$|h|=n-3$ and $|f_k|=n-1$
	are odd. Thus, the right-hand side of
	\eqref{eq:Theta-normalized} is odd precisely when both elements belong
	to $\{a_{12},a_{13},a_{23}\}.$
	Any two distinct vectors among $a_{12},a_{13},a_{23}$ satisfy
	$|E_{a_{rs}}\cap E_{a_{pq}}|=1.$
	On the other hand,
	$|E_{a_{rs}}\cap E_h|=0,
	|E_{a_{rs}}\cap E_{f_k}|=2,$
	for $rs\in\{12,13,23\}$ and $4\leq k\leq n$. Moreover,
	$|E_h\cap E_{f_k}|=n-4,$
	and, for distinct $k,l\geq4$,
	$|E_{f_k}\cap E_{f_l}|=n-2.$
	Since $n$ is even, all of these latter quantities are even.
	Therefore, \eqref{eq:Theta-normalized} holds for every pair of
	distinct elements of $T$, and hence $S$ has the $\Theta$-property.
	
	Finally, suppose that $n=4$. Let
	$f_k=J+e_k$, where $1\leq k\leq4,$
	so that
	$T=\{f_1,f_2,f_3,f_4,0\}.$
	Each $f_k$ has odd weight, and for $k\neq l$,
	$|E_{f_k}\cap E_{f_l}|=2.$
	Also $E_0=\emptyset$, and hence $|E_0\cap E_{f_k}|=0$. Thus \eqref{eq:Theta-normalized} holds for every pair of distinct elements of $T$, so this set has the $\Theta$-property.

	The upper bound $M_\Theta(n)\leq2n+1$ follows from \Cref{thm:upper}, while the constructions above attain this bound.
\end{proof}

For completeness, $M_\Theta(0)=1$, $M_\Theta(1)=2$, and $M_\Theta(2)=3$. For $n=2$, the set $\{e_1,e_2,0\}$ has the $\Theta$-property, whereas the full set $\F_2^2$ does not: the three pairwise disjoint sets $\{0\}$, $\{e_1,e_2,e_1+e_2\}$, and $\emptyset$ have the same sum and two of them have odd size.

\begin{proposition}\label{prop:Theta-dim4-unique}
Every $9$-element subset of $\mathbb Z_2^4$ with the $\Theta$-property is linearly equivalent to
$\{e_1,e_2,e_3,e_4,0,J+e_1,J+e_2,J+e_3,J+e_4\},$
where $J=e_1+e_2+e_3+e_4.$
In particular, every such set contains $0$.
\end{proposition}

\begin{proof}
A $9$-element subset of $\mathbb Z_2^4$ necessarily spans $\mathbb Z_2^4$. After applying a linear automorphism, assume that it contains the standard basis $B=\{e_1,e_2,e_3,e_4\}$, and set $S=B\cup T$, where $|T|=5$. For $t\in T$, let $E_t\subseteq B$ be the support of $t$. By \Cref{lem:Theta-q-linear}, for distinct $t,u\in T$,
\[
|E_t\cap E_u|=|D_t\cap D_u|\equiv |D_t| |D_u|\equiv(1+|t|)(1+|u|)\pmod2.
\]

Suppose first that $0\notin T$. Since all weight-1 vectors already belong to $B$, every element of $T$ has weight $2$, $3$, or $4$. The vector $J$ of weight $4$ is incompatible with every other possible element of $T$: with a vector of weight $2$ the left-hand side above is even while the right-hand side is odd, and with a vector of weight $3$ the left-hand side is odd while the right-hand side is even. Hence, $J\notin T$.

For two vectors of weight $2$, the condition says that their supports must intersect in exactly one coordinate. A vector of weight $3$ is compatible with a vector of weight $2$ precisely when its support contains the support of the weight-$2$ vector, while any two distinct vectors of weight $3$ are compatible. If $T$ contains no vector of weight $2$, then it contains at most the four vectors of weight $3$. If it contains one vector of weight $2$, at most two weight-$3$ vectors are compatible with it. If it contains two weight-$2$ vectors, their supports have union of size $3$, so at most one weight-$3$ vector is compatible with both. Finally, a pairwise-intersecting family of $2$-subsets of a $4$-set has size at most $3$; with three such vectors, at most one weight-$3$ vector can be compatible with all of them. In every case $|T|\leq4$, a contradiction.

Therefore, $0\in T$. Applying the displayed parity condition to $0$ and any $t\in T\setminus\{0\}$ gives
$0\equiv1+|t|\pmod2,$
so every such $t$ has odd weight. Since all weight-1 vectors already belong to $B$, every element of $T\setminus\{0\}$ has weight $3$. There are exactly four weight-$3$ vectors in $\mathbb Z_2^4$, namely $J+e_1,\ldots,J+e_4$, and all four must occur because $|T|=5$. This proves the claim.
\end{proof}

\subsection{The additional Sidon restriction}\label{sec:SmallSidon}

We now ask how much is lost if $S$ is also required to be Sidon. Define
\[
S_{\rm Sid}(n)=\max\bigl\{|S|:S\subseteq\mathbb Z_2^n\text{ is Sidon, generates }\mathbb Z_2^n,\text{ and has the }\Theta\text{-property}\bigr\}.
\]
By \Cref{thm:upper}, $S_{\rm Sid}(n)\leq M_\Theta(n)\leq2n+1$. Thus $2n+1$ is the absolute bound imposed by the $\Theta$-property itself, independently of Sidon. The Sidon extremal problem asks whether this absolute bound can still be attained after forbidding repeated pair sums.

Observe that $S$ has no dependency relation of size $2$; the Sidon property is equivalent to the absence of dependency relations of size $4$. If, in addition, $S$ has the $\Theta$-property, then there can be at most one dependency relation of size $3$. Indeed, if $D_1,D_2\in\mathcal D$ are two distinct dependency relations of size $3$, then \Cref{lem:Theta-q-linear} implies that $|D_1\cap D_2|$ is odd. Since they are distinct, $|D_1\cap D_2|=1$, and then $D_1\oplus D_2$ is a dependency relation of size $4$, contradicting the Sidon property.

\subsection{Exact Sidon values in small dimensions}

For $3\leq n\leq9$, the Sidon restriction prevents the absolute bound $2n+1$ from being attained. An exhaustive computer search gives the exact values of $S_{\rm Sid}(n)$ which are given in \Cref{tab:sid-values}. In each row, the displayed set $T$, together with the standard basis, gives a Sidon set with the $\Theta$-property of the stated size. We also include an example of size $21$ for $n=10$.

\begin{table}[htbp]
	\centering
	\begin{tabular}{c|c|l}
		$n$ & $S_{\rm Sid}(n)$ & $T$ \\ \midrule
		0  & 1  & $\{0\}$ \\
		1  & 2  & $\{0\}$ \\
		2  & 3  & $\{3\}$ \\
		3  & 4  & $\{6\}$ \\
		4  & 5  & $\{15\}$ \\
		5  & 7  & $\{30,31\}$ \\
		6  & 8  & $\{60,62\}$ \\
		7  & 10 & $\{113,126,127\}$ \\
		8  & 12 & $\{202,241,252,254\}$ \\
		9  & 15 & $\{165,329,399,497,510,511\}$ \\
		10 & 21 & $\{75,151,184,302,369,452,549,604,738,786,905\}$
	\end{tabular}
	\caption{Values of $S_{\rm Sid}(n)$ and the corresponding sets $T$.}
	\label{tab:sid-values}
\end{table}

Here an integer $a$ denotes the binary vector $\sum_{j=1}^n a_j e_j$, where $a=\sum_{j=1}^n a_j2^{j-1}$; thus the least significant bit corresponds to the first coordinate. The exhaustive search is performed after fixing the standard basis $B=\{e_1,\ldots,e_n\}$ and setting $S=B\cup T$. Candidate elements $t$ are added recursively. At each step, $t$ is retained only if the new pair sums $t+s$, $s\in S$, are distinct from all existing pair sums, and if, for every previously chosen $u\in T$,
\(
|E_t\cap E_u|\equiv (|t|+1)(|u|+1)\pmod2.
\)
By \Cref{lem:Theta-q-linear}, these are precisely the Sidon and $\Theta$-conditions.

\subsection{Sidon constructions attaining the absolute bound}

It is notable that the additional Sidon condition eventually causes no loss in extremal size. We prove this by using binary self-dual codes. We recall only the terminology needed here and refer to \cite{HP03} for general background.

A binary linear code of length $N$ is a subspace of $\F_2^N$. The weight $|x|$ of a vector $x$ is the number of its nonzero coordinates, and the minimum distance of a nonzero linear code $C$ is
\[
d(C)=\min\{|x|:x\in C,\ x\neq0\}.
\]
A code of length $N$, dimension $k$, and minimum distance $d$ is called an $[N,k,d]$ code. Its dual is
\(
C^\perp=\{y\in\F_2^N:x\cdot y=0\text{ for every }x\in C\},
\).

A code $C$ is said to be self-dual if $C=C^\perp$.

If equality holds in \Cref{thm:upper}, then $|S|=2n+1$ and $\widehat{\mathcal D}$ has dimension $(|S|+1)/2$, so the self-orthogonal space $\widehat{\mathcal D}$ is in fact self-dual. This suggests obtaining extremal Sidon sets from binary self-dual codes.

\begin{proposition}\label{prop:puncture}
Let $C$ be a binary self-dual $[2r,r,d]$ code with $d\geq6$. Then there exists a Sidon set $S\subseteq\F_2^{r-1}$ which generates $\F_2^{r-1}$, has the $\Theta$-property, and satisfies $|S|=2r-1$.
\end{proposition}

\begin{proof}
Let $K$ be the code obtained from $C$ by deleting the last coordinate ($K\subseteq\F_2^{2r-1}$). Since $d(C)\geq6$, the projection of $C$ onto $K$ is injective, hence $\dim K=r$. Moreover, $d(K)\geq5$.

Every codeword of a binary self-dual code has even weight. Therefore the value at the deleted coordinate of a code word $x$ of $C$ is determined by the parity of the remaining coordinates. In other words if we add a new (last) coordinate to each code word of $K$, and set values of the new coordinate for $x$ to be $|x|$, then the result is $C$: 

\[
C=\widehat K=\{(x,|x|\bmod2):x\in K\}.
\]
Let $H$ be a parity-check matrix for $K$, that is to say $K=\{x \mid Hx=0\}$. Since $K$ has length $2r-1$ and dimension $r$, the matrix $H$ has rank $r-1$. Let $S$ be the set of its columns. Then $S\subseteq\F_2^{r-1}$, $|S|=2r-1$, and $S$ generates $\F_2^{r-1}$.

The dependency relations among the elements of $S$ are precisely the codewords of $K$. Since $d(K)\geq5$, there are no dependency relations of size $1$, $2$, $3$, or $4$. In particular, the columns of $H$ are nonzero and distinct and no two distinct pairs of columns have the same sum. Hence, $S$ is Sidon.

Finally, since $C=\widehat K$ is self-dual, it is self-orthogonal. Thus, for $x,y\in K$,
\[
0=(x,|x|\bmod2)\cdot(y,|y|\bmod2),
\]
and therefore $x\cdot y\equiv |x||y|\pmod2$. In terms of dependency relations this is precisely the condition of \Cref{lem:Theta-q-linear}, so $S$ has the $\Theta$-property.
\end{proof}

By \cite[Theorem~4]{CS90}, a binary self-dual code with minimum distance at least $6$ exists for every even length at least $22$. Thus, for every $r\geq11$, there exists a binary self-dual code with parameters $[2r,r,d]$ and $d\geq6$.

\begin{theorem}\label{thm:S-large}
For every $n\geq10$,
$S_{\rm Sid}(n)=M_\Theta(n)=2n+1.$
\end{theorem}

\begin{proof}
Let $n\geq10$, and set $r=n+1$. Then $r\geq11$, so there exists a binary self-dual code $C\subseteq\F_2^{2r}$ with minimum distance at least $6$. By \Cref{prop:puncture}, this gives a Sidon set $S\subseteq\F_2^{r-1}=\F_2^n$ which generates $\F_2^n$, has the $\Theta$-property, and has size $|S|=2r-1=2n+1$. The reverse inequality follows from \Cref{thm:upper}. Hence $S_{\rm Sid}(n)=2n+1$, and \Cref{thm:Theta-extremal} gives $M_\Theta(n)=2n+1$ as well.
\end{proof}

Thus the unrestricted and Sidon extremal problems have the same answer for every $n\geq10$, even though they differ in several small dimensions. In particular, the Sidon condition eventually imposes no loss in the maximum possible size of a generating set with the $\Theta$-property.

\section{Concluding remarks}

Our main focus was the structural study of cube-like graphs admitting certain complex signings. In the following subsection, we present the motivation that has led to this study and note how other settings may help with similar structural results. We then note a completely different direction of interest arising from this study.

\subsection{A generalization of the independence number}

Given a graph $G$ and an integer $j$ with $0\leq j\leq\Delta(G)$, we define $\alpha_j(G)$ to be the maximum order of a vertex set $V_j\subseteq V(G)$ such that
\(
    \Delta(G[V_j])\leq j,
\)
 where $G[V_j]$ is the subgraph induced by $V_j$.\\

By definition,
\(
    \alpha_0(G)\leq\alpha_1(G)\leq\cdots\leq\alpha_{\Delta(G)}(G).
\)
Moreover, $\alpha_0(G)=\alpha(G)$ is the classical independence number and $\alpha_{\Delta(G)}(G)=|V(G)|$.\\

Huang's result \cite{H19} can be restated as
\(
    \alpha_{\lceil\sqrt n\rceil-1}(H_n)=2^{n-1}.
\)\\

His proof uses an edge weighting ($\sigma_n$) resulting in only two eigenvalues then applies the Cauchy interlacing theorem; though it can be presented using only linear dependencies (see \cite{LNSW26}). Constructions in \cite{CFGS88} show that
\(
    \alpha_{\lceil\sqrt n\rceil}(H_n)\geq2^{n-1}+1,
\)
but the exact value of $\alpha_{\lceil\sqrt n\rceil}(H_n)$ is not known in general.

Alon and Zheng \cite{AZ20} extended Huang's spectral technique to cube-like graphs. In particular, their argument implies that every induced subgraph of a cube-like graph $Q_S$ on more than $2^{n-1}$ vertices has maximum degree at least $\sqrt{|S|}$. In our notation,
\(
    \alpha_{\lceil\sqrt{|S|}\rceil-1}(Q_S)\leq2^{n-1}.
\)

In this work, we show that the unitary-signing technique applies to every generating set $S$ with the $\Theta$-property; for Sidon sets, this condition is also necessary.

Although having very few distinct eigenvalues is useful for this problem, it is not the whole story. For example, the independence number of $PC(2k)$ can be bounded using the inertia of its adjacency matrix. The spectrum of the projective cube is well known; see, for example, \cite[Section~9.2D]{BCN89} and the references therein.  The exact value of the independence number is recorded in MathWorld \cite{MathWorldFoldedCube}, where it is attributed to an unpublished note of Godsil \cite{Godsil06}.

\begin{theorem}\label{thm:spectrum-PC}
For $k\geq1$, the distinct eigenvalues of $PC(2k)$ are
\(
    \lambda_r=2k+1-4r,\\
     0\leq r\leq k,
\)
where the multiplicity of $\lambda_r$ is
\(
    \binom{2k+1}{2r}.
\)
\end{theorem}

It follows that the numbers of positive and negative eigenvalues of $PC(2k)$ are
\[
    n_+\bigl(PC(2k)\bigr)
    =2^{2k-1}+(-1)^k\binom{2k-1}{k-1},
\quad
    n_-\bigl(PC(2k)\bigr)
    =2^{2k-1}-(-1)^k\binom{2k-1}{k-1}.
\]
Since none of the eigenvalues is zero, the inertia bound, which follows from the Cauchy interlacing theorem, gives
\(
    \alpha(PC(2k))
    \leq \min\{n_+,n_-\}
    =2^{2k-1}-\binom{2k-1}{k-1}.
\)

To see that the found is tight, choose a vertex $v$ of $PC(2k)$. If $k$ is even, the vertices at odd distance from $v$ form an independent set of this order; if $k$ is odd, the vertices at even distance from $v$ form such an independent set. Hence,
\(
    \alpha(PC(2k))
    =2^{2k-1}-\binom{2k-1}{k-1}.
\)

%
%
%
%
%
%
%
%
%
%
%
%

\subsection{Further questions on complex-signed graphs}

As noted in the introduction, complex-signed graphs may equivalently be viewed as gain graphs over the fourth roots of unity and hence as a special case of complex unit gain graphs; see, for example, \cite{SK21}. The signed-graph terminology used here is convenient for emphasizing the real/imaginary and positive/negative types of cycles, switching, and the associated Hermitian adjacency matrix.

Theories of minors, colorings, and homomorphisms of signed graphs have proved fruitful, in particular by strengthening connections between minor theory and coloring; see, for example, \cite{NRS15,JMNNQ26}. It is natural to ask analogous questions in the present four-valued setting. Such questions are also related to homomorphisms of mixed graphs; see, for example, \cite{AM98,NR00}.

For this purpose, a homomorphism from a complex-signed graph $(G,\sigma)$ to $(H,\pi)$ is a map $f:V(G)\to V(H)$ for which there is a signing $\sigma'$ switching equivalent to $\sigma$ such that $f$ preserves adjacency and $\pi(f(x)f(y))=\sigma'(xy)$
for every oriented adjacency $xy$ of $G$, with the usual convention that reversing an arc conjugates its assigned value. This leads to the following problem.

\begin{problem}
What is the smallest order of a complex-signed graph with no multiple edges to which every planar complex-signed graph admits a homomorphism?
\end{problem}

	\section*{Acknowledgements} Meirun Chen is supported by Fujian Provincial Department of Science and Technology (2024J011197). Reza Naserasr has received support under the program ``Investissement d'Avenir" launched by the French Government and implemented by ANR, with the reference ``ANR‐18‐IdEx‐0001" as part of its program ``Emergence". 
	
		\section*{Declaration on the use of AI} During the preparation of this paper, ChatGPT and Kimi, on basic plans, were used for proofreading, language editing, improving clarity and fluency, and computational work in \Cref{sec:SmallSidon}. Gemini, on a free plan, was used to create or edit TikZ codes for figures. The mathematical ideas, proofs, and all technical mathematical content are the authors' own work.

	\bibliographystyle{plain}
	\bibliography{references}

\end{document}